\documentclass[11pt]{article}
\usepackage{amsmath, amssymb, amscd, amsthm, amsfonts,mathrsfs}
\usepackage{graphicx}
\usepackage{hyperref}
\usepackage{authblk}
\usepackage{refcount}
\usepackage{float}
\usepackage{booktabs}
\usepackage{enumitem}

\newtheorem{theorem}{Theorem}
\newtheorem{corollary}{Corollary}
\newtheorem{definition}{Definition}
\newtheorem{lemma}{Lemma}
\newtheorem{remark}{Remark}
\newtheorem{example}{Example}

\title{Peakon solutions and analytical properties for the Camassa-Holm type equations with quadratic nonlinearities}
\author{Yonghong Chen}
\author{Zhijun Qiao\thanks{Corresponding authors \  Emails: zhijun.qiao@utrgv.edu \  mxzhu@qfnu.edu.cn}}
\author[2]{Mingxuan Zhu$^\ast$}
\affil[1]{School of Mathematical and Statistical Sciences, University of Texas Rio Grande Valley, Edinburg,  TX 78539,USA}
\affil[2]{Department of Mathematics, Qufu Normal University, Qufu, 273165, P. R. China}

\date{}

\begin{document}

\maketitle

\begin{abstract}
In this paper, we derive the multi-peakon dynamical system of a class of Camassa-Holm-type equations with quadratic
nonlinearities. We also consider the analytical properties for the Cauchy problem. Firstly, we establish local well-posedness of solutions in Besov spaces and then 
provide the blow-up criteria. Subsequently, we impose appropriate sufficient conditions on the initial data to guaranty that the corresponding solution either exists globally or blows up in a finite time. Finally, we prove the ill-posedness in the Besov space $B_{2,\infty}^{3/2}$ by utilizing the non-traveling wave solutions.
\end{abstract}

\section{Introduction}

A class of fundamental equations in the study of nonlinear waves is characterized by the following form:
\begin{equation}
u_t + P(u, u_x) + (1-\partial_x^2)^{-1}Q(u,u_x) + \partial_x(1-\partial_x^2)^{-1}R(u,u_x) = 0,
\end{equation}
where $P, Q,$ and $R$ are bivariate polynomials of degree $n$, and $(1-\partial_x^2)^{-1}$ denotes the convolution with the Helmholtz kernel $G(x) = \frac{1}{2}e^{-|x|}$ on $L^2(\mathbb{R})$. Representative examples include:

\begin{itemize}
    \item \textbf{Camassa-Holm equation:}
    \begin{equation}
    \label{Camassa-Holm}
            u_t + u u_x + \partial_x (1-\partial_x^2)^{-1} \left( u^2 + \frac{1}{2} u_x^2 \right) = 0
    \end{equation}
    \item \textbf{Degasperis–Procesi equation:}
    \begin{equation}
    u_t + u u_x + \partial_x (1-\partial_x^2)^{-1} \left( \frac{3}{2}u^2 \right) = 0
    \end{equation}
    \item \textbf{Novikov equation:}
    \begin{equation}
    u_t + u^2 u_x + \partial_x (1-\partial_x^2)^{-1} \left( u^3 + \frac{3}{2} u u_x^2 \right) + (1-\partial_x^2)^{-1} \left( \frac{1}{2} u_x^3 \right) = 0 
    \end{equation}
    \item \textbf{Fokas-Olver-Rosenau-Qiao equation:}
    \begin{equation}
    u_t + u^2 u_x - \frac{1}{3} u_x^3 + \partial_x (1-\partial_x^2)^{-1} \left( \frac{2}{3} u^3 + u u_x^2 \right) + (1-\partial_x^2)^{-1} \left( \frac{1}{3} u_x^3 \right) = 0
    \end{equation}
\end{itemize}

A common feature of these models is their admission of peakon solutions. Formally, any function of the form 
\begin{equation*}
u(x,t) = p(t)e^{-|x-q(t)|}
\end{equation*}
can be regarded as a candidate peakon. To the best of our knowledge, waves similar to peakons first appeared in the context of the Fornberg--Whitham equation \cite{fornberg1978numerical}:
\begin{equation}
u_t + \frac{3}{2}u u_x + \partial_x (1-\partial_x^2)^{-1} u = 0.
\end{equation}
This equation admits the traveling wave solution
\begin{equation}
u(x,t) = \frac{8}{9} \, e^{-\frac{1}{2} \, |x - \frac{4}{3} t|}.
\label{peakon1}
\end{equation}

The Fornberg--Whitham equation also exhibits the phenomenon of wave breaking \cite{haziot2017wave}, in a manner similar to the Camassa--Holm equation. However, in contrast to Camassa--Holm type equations, which possess a whole family of peakon solutions, the solution \eqref{peakon1} is essentially the unique peakon of the Fornberg--Whitham equation. This distinction suggests, at least heuristically, that the richness of peakon structures may be closely related to the integrability of the underlying equation.

The Camassa–Holm (CH) equation \cite{camassa1993integrable} was derived by Camassa and Holm in the study of shallow water wave regimes. They established its physical relevance, gave the Lax pair to guaranty the integrability of the equation, and showed that the CH equation admits peakon solutions, the most breakthrough achievement in soliton theory. Later in 1995 Fokas \cite{fokas1995physicad} generalized the CH model to a large class of physically important integrable equations.    
Based on those foundational results, Qiao \cite{qiao2003camassa} further developed the Camassa--Holm hierarchy, investigated its integrable extensions of dimensions $N$ and constructed algebro-geometric solutions on a symplectic submanifold.

In particular, the Camassa–Holm equation is renowned for exhibiting fascinating phenomena: wave breaking, blow-up, and peakon stability. These phenomena play crucial roles in both the physical interpretation and the mathematical structure of the equation.

\begin{enumerate}

    \item \textbf{Wave Breaking and Blow-up}: 
    One of the features of the Camassa–Holm equation is its ability to model wave breaking, 
    where the solution remains bounded while its slope becomes unbounded in finite time. 
    As described in \cite{camassa1993integrable}, a smooth initial condition can break down into a series of peakons 
    by developing vertical slopes at each inflection point with negative gradient, from which derivative discontinuities emerge. This provides a concrete dynamical mechanism for the formation of singularities. 
    Constantin and Escher \cite{constantin1998global,constantin1998wave} later gave a rigorous justification of this phenomenon, 
    showing that such wave breaking occurs even when the energy remains finite. From a mathematical perspective, wave breaking represents a particular form of finite-time blow-up. 
    More generally, the Camassa–Holm equation admits blow-up solutions even for smooth initial data, 
    indicating that singularity formation is an intrinsic feature of the equation.

    \item \textbf{Peakon Stability}: 
    Another crucial aspect of the Camassa–Holm equation is the stability of peakon solutions.
    Constantin and Strauss \cite{constantin2000stability} investigated the orbital stability of these localized waves,
 showing that the peakons persist under small perturbations. Their results highlight the robustness of peakons as
    physically relevant structures in nonlinear shallow water wave dynamics. In \cite{qiao2006peaked}, Qiao and
    Zhang analyze the peaked and smooth soliton solutions of the Camassa–Holm equation and investigate their dynamical
    behavior. This dynamical approach might be applied to Solitary and self-similar solutions of two-component system of nonlinear Schr\"{o}dinger equations \cite{lin2006solitary}. 

\end{enumerate}

Similarly, the Degasperis–Procesi equation, which also admits peakon solutions, was derived using the method of asymptotic
integrability by Degasperis and Procesi \cite{degasperis1999asymptotic}. As an integrable system, it possesses a Lax
pair, a bi-Hamiltonian structure, and an infinite hierarchy of symmetries and conservation laws
\cite{degasperis2002new}.

The Degasperis–Procesi equation also exhibits wave-breaking phenomena, though the conditions triggering this behavior differ
slightly from those of the Camassa–Holm equation. Liu and Yin~\cite{liu2006global} demonstrated that wave breaking occurs
under specific sign conditions imposed on the initial data. Concerning the blow-up phenomenon, Escher, Liu, and
Yin~\cite{escher2006global} established the precise blow-up rate and characterized the blow-up set for strong
solutions to the Degasperis–Procesi equation across a wide range of initial data. Additionally, Liu and Yin~\cite{liu2006global} proved
that the first blow-up in finite time for the Degasperis–Procesi equation manifests itself as wave breaking, with the potential emergence
of shock waves thereafter. In particular, the blow-up rate for the breaking waves in the Camassa–Holm equation is
$-2$~\cite{constantin2000existence}, while for the Degasperis–Procesi equation, it is $-1$~\cite{escher2006global}. Regarding
peakon stability in the Degasperis–Procesi equation, this property was established by Lin and Liu in~\cite{lin2009stability}.

The Holm-Staley \( b \)-family of equations is a one-parameter generalization of the CH and Degasperis–Procesi equations, introduced
by Holm and Staley \cite{holm2003nonlinear} to describe nonlinear shallow water wave dynamics with varying
nonlinear dispersion properties. This family is given by
\begin{equation}
m_t  + m_xu  + bmu_x = 0,\quad m = u - u_{xx},
\end{equation}
where \( b \) is a real parameter.

In particular, for \( b = 2 \), the equation reduces to the Camassa–Holm equation, while for \( b = 3 \), it corresponds to the DP
equation. The Holm-Staley equation \( b \)-family exhibits rich mathematical structures, including a non-canonical
Hamiltonian formulation, and also admits peakon solutions \cite{degasperis2003integrable}.

The purpose of this paper is twofold. First, we show how to derive the $N$-peakon dynamical system from Camassa--Holm-type equations with quadratic nonlinearity. Second, we consider the Cauchy problem for this class of equations, investigating local and global well-posedness, wave breaking, and ill-posedness.

\section{Camassa--Holm Type Equations with Quadratic Nonlinearities}

In this section, we derive the dynamical system of $N$ -peakon in a nonlocal formulation. The convolution-based computations developed by Lundmark in \cite{lundmark2007formation} are adapted here to a general class of Camassa--Holm-type equations with quadratic nonlinearity:
\begin{equation}
    u_t + \lambda_1 u^2 + \lambda_2 u u_x 
    + G \ast (\lambda_3 u^2 + \lambda_4 u_x^2) 
    + \partial_x \left( G \ast (\lambda_5 u^2 + \lambda_6 u_x^2) \right) = 0,
    \label{Nonlocallambda}
\end{equation}
where $G(x) = \tfrac{1}{2}e^{-|x|}$ is the Helmholtz kernel and $\ast$ denotes the convolution with respect to the spatial variable.

As Constantin and Strauss pointed out, the nonlocal formulation provides a natural framework in which peakon solutions of the Camassa–Holm equation can be rigorously interpreted \cite{constantin2000stability}. We extend this point of view to the general class of equations \eqref{Nonlocallambda} and introduce the following definition.

\begin{definition}[N-peakon solution]
A function of the form
\begin{equation}
u(x,t) = \sum_{i=1}^N p_i(t)\,e^{-|x-q_i(t)|}
\end{equation}
is called an $N$-peakon solution of equation \eqref{Nonlocallambda} on $\mathbb{R}\times [0,T)$, for some $T>0$, if it satisfies \eqref{Nonlocallambda} in the weak sense.

Throughout this paper, we introduce the following shorthand notation. We set
\begin{equation}
\mathcal{E}_\ell = e^{-|x-q_\ell|}, \qquad 
\mathcal{E}_{i,j} = e^{-|q_i-q_j|},
\end{equation}
and
\begin{equation}
\sigma_\ell = \operatorname{sgn}(x-q_\ell), \qquad
\sigma_{ij} = \operatorname{sgn}(q_i-q_j),
\end{equation}
where $q_\ell = q_\ell(t)$ for all $\ell$, depending only on the time variable $t$.
\end{definition}

To derive peakon solutions, a key step is to evaluate convolutions with the Helmholtz kernel $G$. 
Remarkably, the convolution of $G$ with expressions of the form $\mathcal{E}_i \mathcal{E}_j$ remains algebraically closed within the class generated by the above notation. 
More precisely, the result can be expressed in terms of $\mathcal{E}_\ell$, $\sigma_\ell$, $\mathcal{E}_{i,j}$, and $\sigma_{ij}$.

A direct computation yields the following identities.

\begin{lemma}\label{lemma1}
\begin{equation}
\begin{aligned}
    G\ast \Big( \mathcal{E}_i\mathcal{E}_j\Big) & = \frac 13 (  -\mathcal{E}_i \mathcal{E}_j+\mathcal{E}_{i,j}\mathcal{E}_i + \mathcal{E}_{i,j}\mathcal{E}_j) + \frac{2\sigma_{ij}}{3}(\sigma_j\mathcal{E}_i\mathcal{E}_j -\sigma_i\mathcal{E}_j\mathcal{E}_i+\mathcal{E}_{i,j}\sigma_i\mathcal{E}_i - \mathcal{E}_{i,j}\sigma_j\mathcal{E}_j  )\\
    G\ast \Big( \sigma_i\sigma_j\mathcal{E}_i\mathcal{E}_j\Big) & =  \frac 13 (  -\mathcal{E}_i \mathcal{E}_j+\mathcal{E}_{i,j}\mathcal{E}_i + \mathcal{E}_{i,j}\mathcal{E}_j) - \frac{\sigma_{ij}}{3}(\sigma_j\mathcal{E}_i\mathcal{E}_j -\sigma_i\mathcal{E}_j\mathcal{E}_i+\mathcal{E}_{i,j}\sigma_i\mathcal{E}_i - \mathcal{E}_{i,j}\sigma_j\mathcal{E}_j  )
\end{aligned}
\end{equation}
In particular,
\begin{equation}
 G\ast (\mathcal{E}_i^2) = \frac{1}{3}(-\mathcal{E}_{i}^2 +2\mathcal{E}_{i}).
\end{equation}
\end{lemma}

\begin{theorem}\label{theorem1}
    If an equation of the form \eqref{Nonlocallambda} satisfies the parameter conditions
    \begin{equation*}
    \lambda_3 = \lambda_1,\quad \lambda_4 = 2\lambda_1,\quad \lambda_5 = \frac32\lambda_2 - \lambda_6,
    \end{equation*}
    then it admits an $N$-peakon solution, where the time evolution of the amplitude $p_i$ and positions $q_i$ is governed by the dynamical system:
\begin{equation}\label{ODEsystems}
\begin{aligned}
    p_i^\prime &= 2\sum_{j=1}^N \left( \frac{2\lambda_5 - \lambda_6}{3}\sigma_{ij} - \lambda_1\right) p_i p_j \mathcal{E}_{i,j}, \\
    q_i^\prime &= 2\sum_{j=1}^N  \left( \frac{\lambda_5 + \lambda_6}{3} \right) p_j\mathcal{E}_{i,j}.
\end{aligned}
\end{equation}
\end{theorem}

\begin{proof}
    By the $N$-peakon ansatz and Lemma \ref{lemma1}, we obtain:
    \begin{align*}
    G \ast \Big(\lambda_3 u^2 + \lambda_4 u_x^2 \Big) = \sum_{i=1}^N \sum_{j=1}^N p_i p_j \Bigg[ &\frac{\lambda_3 + \lambda_4}{3} \Big( -\mathcal{E}_i \mathcal{E}_j + \mathcal{E}_{i,j}\mathcal{E}_i + \mathcal{E}_{i,j}\mathcal{E}_j \Big) \\
    + &\frac{(2\lambda_3 - \lambda_4)\sigma_{ij}}{3} \Big( \sigma_j\mathcal{E}_i\mathcal{E}_j - \sigma_i\mathcal{E}_j\mathcal{E}_i + \mathcal{E}_{i,j}\sigma_i\mathcal{E}_i - \mathcal{E}_{i,j}\sigma_j\mathcal{E}_j \Big) \Bigg].
    \end{align*}
    Taking derivative yields:
    \begin{align*}
    \partial_x G \ast \Big(\lambda_5 u^2 + \lambda_6 u_x^2 \Big) = \sum_{i=1}^N \sum_{j=1}^N p_i p_j \Bigg[ &\frac{\lambda_5 + \lambda_6}{3} \Big( \sigma_i\mathcal{E}_i \mathcal{E}_j +\sigma_j\mathcal{E}_i \mathcal{E}_j -\sigma_i \mathcal{E}_{i,j}\mathcal{E}_i - \mathcal{E}_{i,j}\sigma_j\mathcal{E}_j \Big) \\
    + &\frac{(2\lambda_5 - \lambda_6)\sigma_{ij}}{3} \Big( -\mathcal{E}_{i,j} \mathcal{E}_i + \mathcal{E}_{i,j}\mathcal{E}_j \Big) \Bigg].
    \end{align*}
    Since $i$ and $j$ are dummy indices, we can use symmetry to rewrite the combined expression under the double summation $\sum_{i,j}$ as:
    \begin{align*}
        & \left( \lambda_1 - \frac{\lambda_3+ \lambda_4}{3}\right)\mathcal{E}_i \mathcal{E}_j + 2 \left(  \frac{\lambda_3 + \lambda_4}{3} - \frac{(2\lambda_5 - \lambda_6)\sigma_{ij}}{3}\right) \mathcal{E}_{i,j}\mathcal{E}_i \\
        + & 2 \left( - \frac{(2\lambda_3 - \lambda_4)\sigma_{ij}}{3} + \frac{2\lambda_5 + 2\lambda_6 - 3\lambda_2}{6}\right) \sigma_i\mathcal{E}_i\mathcal{E}_j + 2\left(  \frac{(2\lambda_3 - \lambda_4)\sigma_{ij}}{3} - \frac{\lambda_5 + \lambda_6}{3}\right) \mathcal{E}_{i,j}\sigma_i\mathcal{E}_i.
    \end{align*}
    Under the assumed parameter conditions,
    \begin{equation*}
    \lambda_1 - \frac{\lambda_3+ \lambda_4}{3} = - \frac{2\lambda_3 - \lambda_4}{3} = \frac{2\lambda_5 + 2\lambda_6 - 3\lambda_2}{6} = 0,
    \end{equation*}
    the peakon dynamical system is derived by:
    \begin{align*}
        p_i^\prime &= 2\sum_{j=1}^N \left(  \frac{(2\lambda_5 - \lambda_6)\sigma_{ij}}{3} - \frac{\lambda_3 + \lambda_4}{3}\right)p_i p_j \mathcal{E}_{i,j}, \\
        q_i^\prime &= 2\sum_{j=1}^N  \left( \frac{\lambda_5 + \lambda_6}{3}- \frac{(2\lambda_3 - \lambda_4)\sigma_{ij}}{3}\right) p_j\mathcal{E}_{i,j}.
    \end{align*}
\end{proof}

We now show that Theorem \ref{theorem1} encompasses the standard peakon constructions for a broad class of Camassa--Holm type equations, including the Camassa–Holm equation, the Degasperis–Procesi equation, and the Holm--Staley equation.
\begin{corollary}
Camassa--Holm equation \eqref{Camassa-Holm} admits the following two-peakon dynamical system:
\begin{equation*}
\begin{cases}
    p_1' = p_1 p_2 \sigma_{12} \mathcal{E}_{1,2}, \\
    p_2' = -p_1 p_2 \sigma_{12} \mathcal{E}_{1,2}, \\
    q_1' = p_1 + p_2 \mathcal{E}_{1,2}, \\
    q_2' = p_2 + p_1 \mathcal{E}_{1,2},
\end{cases}
\end{equation*}
where $\mathcal{E}_{1,2} = e^{-|q_1 - q_2|}$ and $\sigma_{12} = \operatorname{sgn}(q_1 - q_2)$.
\end{corollary}

\begin{corollary}
    Equation
    \begin{equation}
    \label{XiaQiao}
    u_t + \frac{1}{2} u^2 + u u_x + G \ast \left( \frac{1}{2}u^2 + u_x^2 \right) +  \partial_x \left( G \ast  \left( \frac{1}{2} u^2
+ u_x^2 \right) \right) = 0    
    \end{equation}
    admits the following two-peakon dynamical system
\begin{equation*}
        \left\{
\begin{aligned}
    p_1' &= -  p_1^2 - p_1 p_2  \mathcal{E}_{1,2},  \\
    p_2' &= -  p_2^2 - p_1 p_2  \mathcal{E}_{1,2},  \\
    q_1' & =  p_1 +  p_2  \mathcal{E}_{1,2}, \\
    q_2' & =  p_2 +  p_1 \mathcal{E}_{1,2}.
\end{aligned}
\right.
\label{2peakonsystem}
\end{equation*}

\end{corollary}
\begin{remark}
Equation \eqref{XiaQiao} arises from the two-component integrable system proposed by Xia and Qiao \cite{xia2015new}. Consequently, it possesses a Lax pair corresponding to its local form, given by
\begin{equation*}
    U = \frac{1}{2} \begin{pmatrix}
    -1 & \lambda m \\
    0 & 1
    \end{pmatrix}, \quad
    V = -\frac{1}{2} \begin{pmatrix}
    \lambda^{-2} & -\lambda^{-1}(u - u_x) + \lambda m u \\
    0 & -\lambda^{-2}
    \end{pmatrix}.
\end{equation*}
\end{remark}

\section{Peakon Solutions}

In this section, we apply Theorem \ref{theorem1} to obtain explicit peakon solutions and study their properties.

\begin{theorem}
    Equation \eqref{Nonlocallambda} has single peakon traveling-wave solutions if and only if
    \begin{equation*}
    \left \{
        \begin{aligned}
            & \lambda_1 = \lambda_3 = \lambda_4 = 0,\\
            & \lambda_2- \frac{2\lambda_5}{3} - \frac{2\lambda_6}{3} = 0.
        \end{aligned}  \right.
    \end{equation*}
\end{theorem}

\begin{proof}
If equation \eqref{Nonlocallambda} admits a single peakon solution $p(t)e^{-|x-q(t)|}$, the corresponding dynamical system must
 satisfy system \eqref{ODEsystems}. To ensure the existence of a traveling wave peakon solution of the
form \( u(x,t) = u(x - ct) \), the velocity \( q'(t) \) must remain constant, which implies
\begin{equation*}
    q' = c = \frac{2(\lambda_5 + \lambda_6)}{3} p.
\end{equation*}
As a result, \( p(t) \) must be constant. Since \( p^2 \neq 0 \) for a nontrivial solution, it follows that $\frac{\lambda_3 + \lambda_4}{3} = 0$.
Furthermore, substituting the single peakon ansatz into equation \eqref{Nonlocallambda} yields the following
constraints:
\begin{equation*}
    \lambda_1 - \frac{\lambda_3+ \lambda_4}{3} = \frac{2\lambda_3 - \lambda_4}{3} = \frac{2\lambda_5 + 2\lambda_6 - 3\lambda_2}{6} = 0.
\end{equation*}
By Theorem \ref{theorem1}, equation \eqref{Nonlocallambda} possesses a single peakon traveling-wave
solution if and only if:
\begin{equation*}
    \lambda_1 = \lambda_3 = \lambda_4 = 0, \quad \lambda_2 - \frac{2\lambda_5}{3} - \frac{2\lambda_6}{3} = 0.
\end{equation*}
\end{proof}

\begin{remark}
Indeed, if $\lambda_1 = \lambda_3 = \lambda_4 = 0$, then equation \eqref{Nonlocallambda} reduces to a conservation form. Therefore, the expression
\begin{equation*}
\frac{\lambda_2}{2}u^2 + G \ast (\lambda_5 u^2 + \lambda_6 u_x^2)
\end{equation*}
can be interpreted as the associated flux.
\end{remark}

\begin{example}
    \label{1peakoncorollary}
Let
\[
u(x,t) =  p(t) \mathrm{e}^{-|x-q(t)|},\quad (x,t) \in \mathbb{R}^2.
\]
If the coefficients in \eqref{Nonlocallambda} satisfy  the following relations:
\[
\lambda_3 = \lambda_1,\quad \lambda_4 = 2\lambda_1,\quad \lambda_5 = \tfrac{3}{2}\lambda_2 - \lambda_6,
\]
then equation \eqref{Nonlocallambda} holds if and only if
\begin{equation}
\left\{
\begin{aligned}
    p' &= -2\lambda_1 p^2, \\
    q' &= \lambda_2 p.
\end{aligned}
\right.
\label{1peakonODE}
\end{equation}
Furthermore, the solution \( u(x,t) \) takes the following form:
\begin{itemize}
    \item \textbf{(Traveling wave solution)} If \( \lambda_1 = 0 \), then
    \[
    u(x,t) = c\, \mathrm{e}^{-|x - \lambda_2 ct|},
    \]
    where \( c \) is an arbitrary constant.
    \item \textbf{(Non-traveling wave solution)} If $\lambda_1 \neq 0$, then
\begin{eqnarray}\label{NTWS}
u(x,t) = \frac{1}{2\lambda_1 t - A}\, \exp\left(-\left| x - \frac{\lambda_2}{2\lambda_1} \ln|2\lambda_1 t - A| - B \right|\right),
\end{eqnarray}
defined on any connected domain where $2\lambda_1 t - A \neq 0$, 
where $A$ and $B$ are arbitrary constants.

\end{itemize}
\end{example}

Although we have derived the explicit $N$-peakon ODE system, numerical methods are recommended to solve peakon solutions, since symbolic computations can be highly time-consuming. We now consider a particular Cauchy problem associated with equation \eqref{XiaQiao}, namely
\begin{equation}
\left\{\begin{array}{l}
u_t + \frac{1}{2} u^2 + u u_x + G \ast \left( \frac{1}{2}u^2 + u_x^2 \right) 
+ \partial_x \left( G \ast \left( \frac{1}{2} u^2 + u_x^2 \right) \right) = 0, 
\quad x\in \mathbb{R},\ t>t_0, \\
u(x, t_0)= \xi_1 \mathrm{e}^{-|x - \eta_1 |} + \xi_2 \mathrm{e}^{-|x - \eta_2 |},\quad \eta_1 < \eta_2,
\end{array}\right.
\label{thetaab111nonlocal}
\end{equation}
for some $t_0 \in \mathbb{R}$. After a tedious computation, the time-shifted solution to the system \eqref{thetaab111nonlocal} is given by
\begin{equation*}
\left\{
    \begin{aligned}
        \tilde p_1(s) & = (C_1 + C_2) \frac{\mu_1 \alpha s^{-\beta} + \mu_2 \beta s^{-\alpha}}{\mu_1 s^\alpha + \mu_2 s^\beta}, \\
        \tilde{p}_2(s) & = \frac{C_2}{s}, \\
        \tilde{q}_1(s) & = \ln\left ( \frac{C_1}{C_1 + C_2} \frac{\mu_1 s^\alpha + \mu_2 s^\beta}{\mu_1 \alpha s^{-\beta}
        + \mu_2 \beta s^{-\alpha}}\right),\\
        \tilde{q}_2(s) & = \ln(s),
    \end{aligned}
\right.
\end{equation*}
where, with $t = \frac{s-C_3}{C_1+C_2}$, the time-shifted variables satisfy
\begin{equation*}
\tilde p_i(s) = p_i(t), \quad \tilde q_i(s) = q_i(t), \quad i=1,2.
\end{equation*}
Hence, a family of two-peakon solutions can be expressed as
\begin{equation*}
    \begin{aligned}
        u(x,t) 
        &= (C_1 + C_2)\,
        \frac{\mu_1 \alpha s^{-\beta} + \mu_2 \beta s^{-\alpha}}{\mu_1 s^\alpha + \mu_2 s^\beta}
        \exp\left(-\left|x - \ln\left( \frac{C_1}{C_1 + C_2} 
        \frac{\mu_1 s^\alpha + \mu_2 s^\beta}{\mu_1 \alpha s^{-\beta} + \mu_2 \beta s^{-\alpha}}\right)\right|\right) \\
        &\quad + \frac{C_2}{s}\, \exp(-|x - \ln s|),
    \end{aligned}
\end{equation*}
where $s=(C_1+C_2)t + C_3$, $\alpha = \frac{C_1}{C_1 + C_2}$ and $\beta = \frac{C_2}{C_1 + C_2}$, 
and $C_1,C_2,C_3 \in \mathbb{R}$ are constants arising from integration, while $\mu_1$ and $\mu_2$ are  constants not simultaneously zero. Moreover, these integration constants are  determined by the initial data $\xi_1,\xi_2,\eta_1,\eta_2$.

Additionally, the distance between the crests of the two peakons is given by
\begin{equation*}
    \tilde{q}_1(s) - \tilde{q}_2(s)
    = \ln\left( \alpha \cdot \frac{\mu_1 s^\alpha + \mu_2 s^\beta}{\mu_1 \alpha s^{\alpha}
    + \mu_2 \beta s^{\beta}} \right).
\end{equation*}

To illustrate the interaction between the two peakons and their time evolution, we plot their dynamics to better understand the behavior of the system. 
Here, we emphasize that all figures are presented in terms of the time-shifted variable $s$.

    \begin{figure}[H]
        \centering

        \begin{minipage}[t]{0.48\textwidth}
            \centering
            \includegraphics[width=\textwidth]{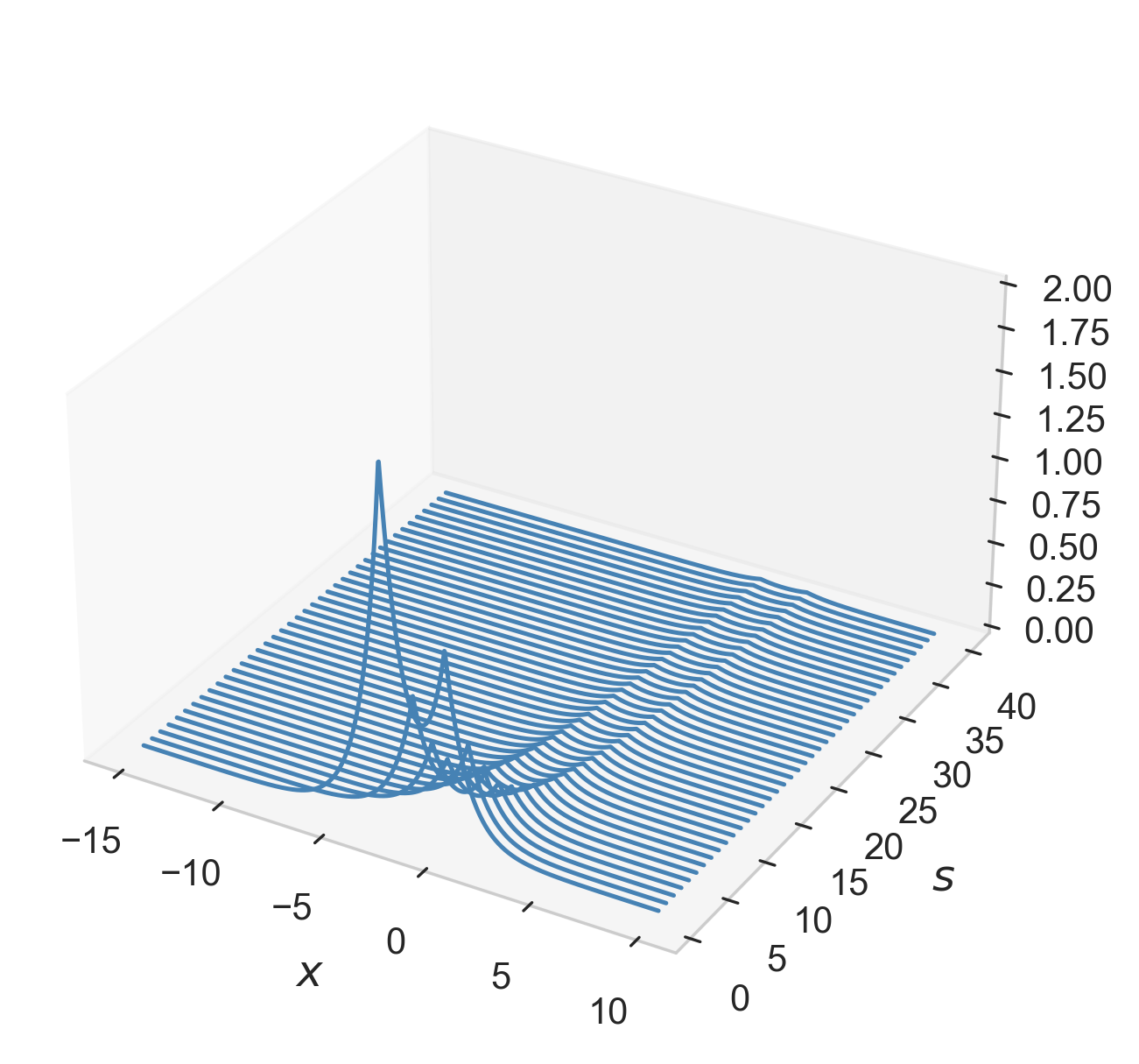}
            \textbf{(A)}
        \end{minipage}%
        \hfill
        \begin{minipage}[t]{0.48\textwidth}
            \centering
            \includegraphics[width=\textwidth]{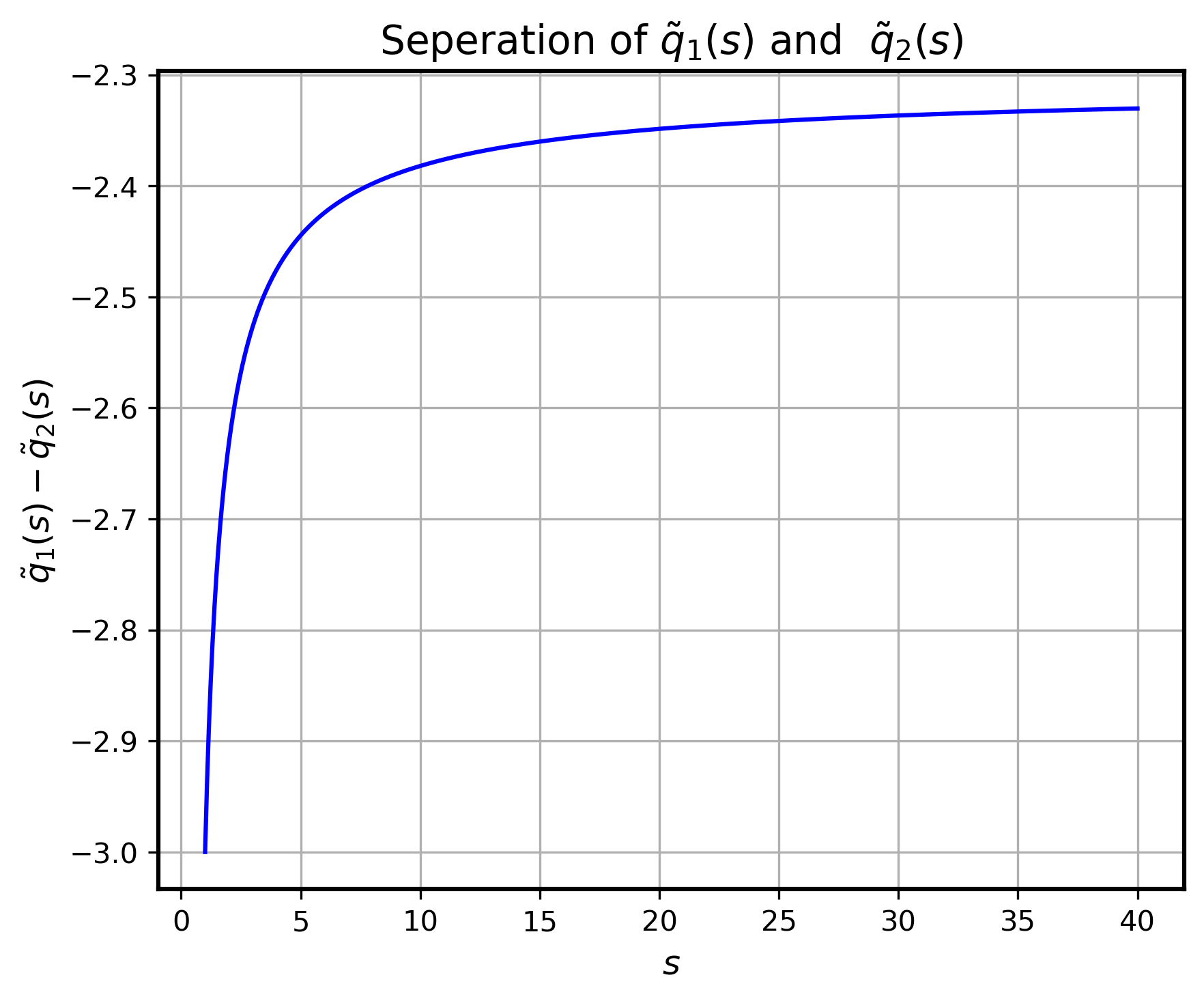}
            \textbf{(B)}
        \end{minipage}
        \caption{
            (A): Two-peakon Evolution (Positive Amplitudes) with Initial Values \(\tilde p_1(1) = 2 \), \(\tilde p_2(1) = 1 \), \(\tilde q_1(1) =
            -3 \), \(\tilde q_2(1) = 0 \), and Parameters \(\mu_1 = 1\), \(\mu_2  = -2 \left(1-\frac{1}{e^3}\right)\), \(\alpha =
            \frac{2}{e^3 + 2 }\), \( \beta = \frac{e^3}{2 + e^3}\), \( C_1 + C_2 = 1+\frac{2}{e^3}>0\).
            (B): Relative Position of Two Peakons.
        }
        \label{fig:labelpp}
    \end{figure}

 \begin{figure}[H]
        \centering
        \begin{minipage}[t]{0.48\textwidth}
            \centering
            \includegraphics[width=\textwidth]{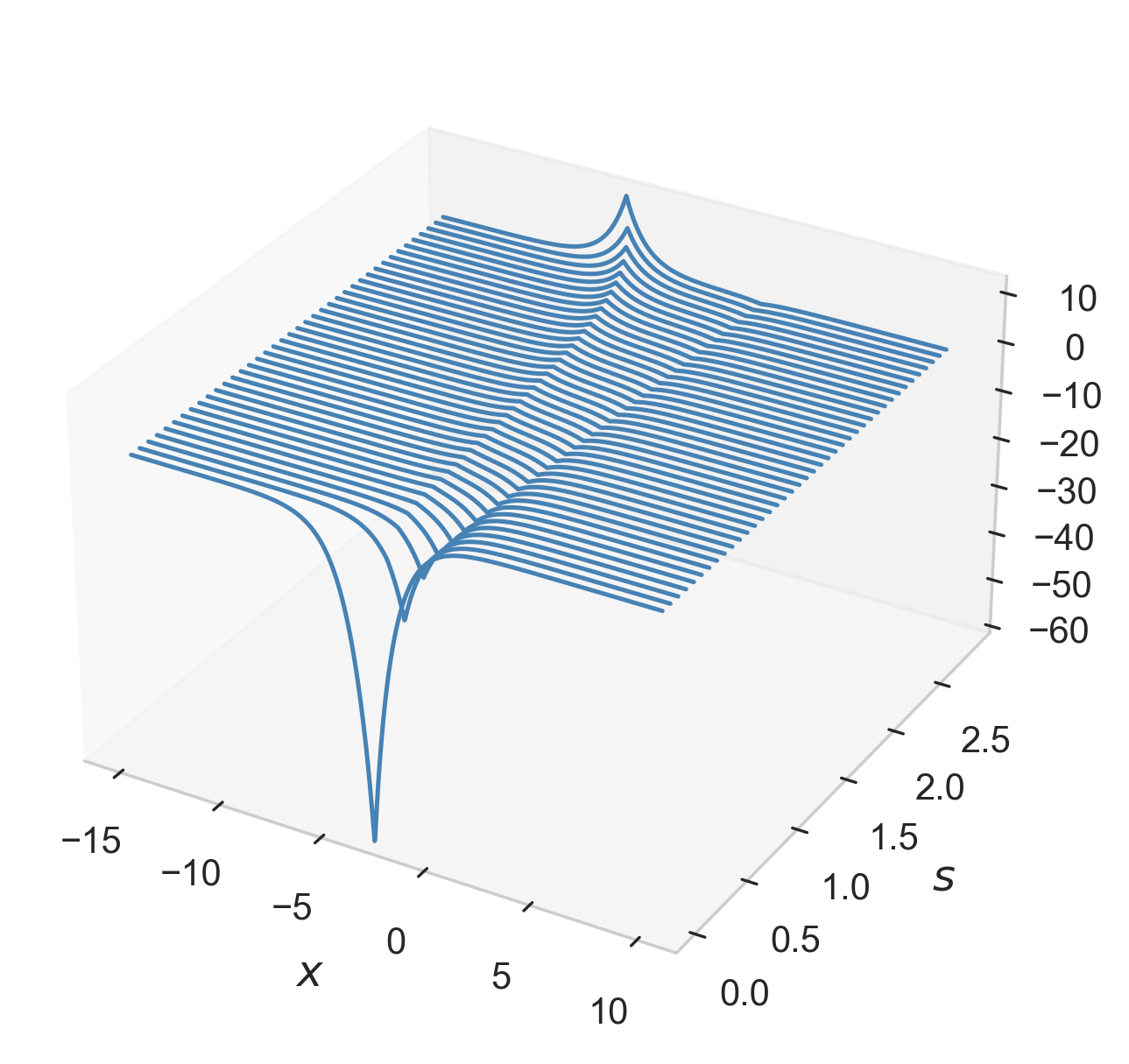}
            \textbf{(A)}
        \end{minipage}%
        \hfill
        \begin{minipage}[t]{0.48\textwidth}
            \centering
            \includegraphics[width=\textwidth]{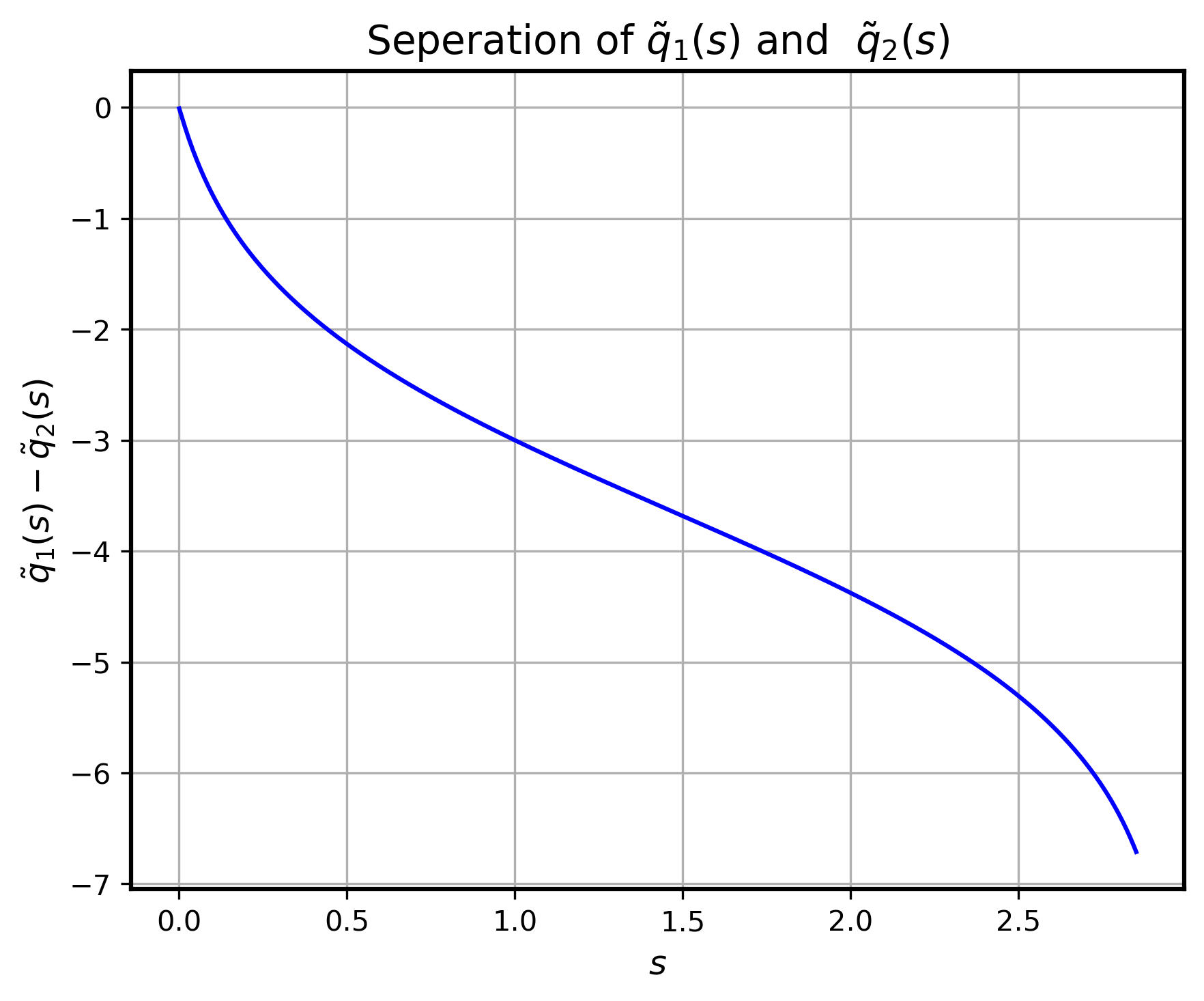}
            \textbf{(B)}
        \end{minipage}
        \caption{
            (A): Peakon-antipeakon Evolution with Initial Values \(\tilde  p_1(1) = 1 \), \(  \tilde p_2(1) = -2 \), \(\tilde  q_1(1) = -3 \), \(
          \tilde  q_2(1) = 0 \) and Parameters \(\mu_1 = 1\), \(\mu_2  = \frac{1}{3} \left(\frac{1}{e^3}-1\right)\),  \(\alpha =
            \frac{1}{1 -2e^3 }\), \( \beta = \frac{2e^3}{2e^3 - 1 }\), \( C_1 + C_2 = \frac{1}{e^3}-2 < 0\).
            (B): Relative Position of Two Peakons.}
            \label{fig:labelpm}
    \end{figure}

\begin{figure}[H]
        \centering
        \begin{minipage}[t]{0.48\textwidth}
            \centering
            \includegraphics[width=\textwidth]{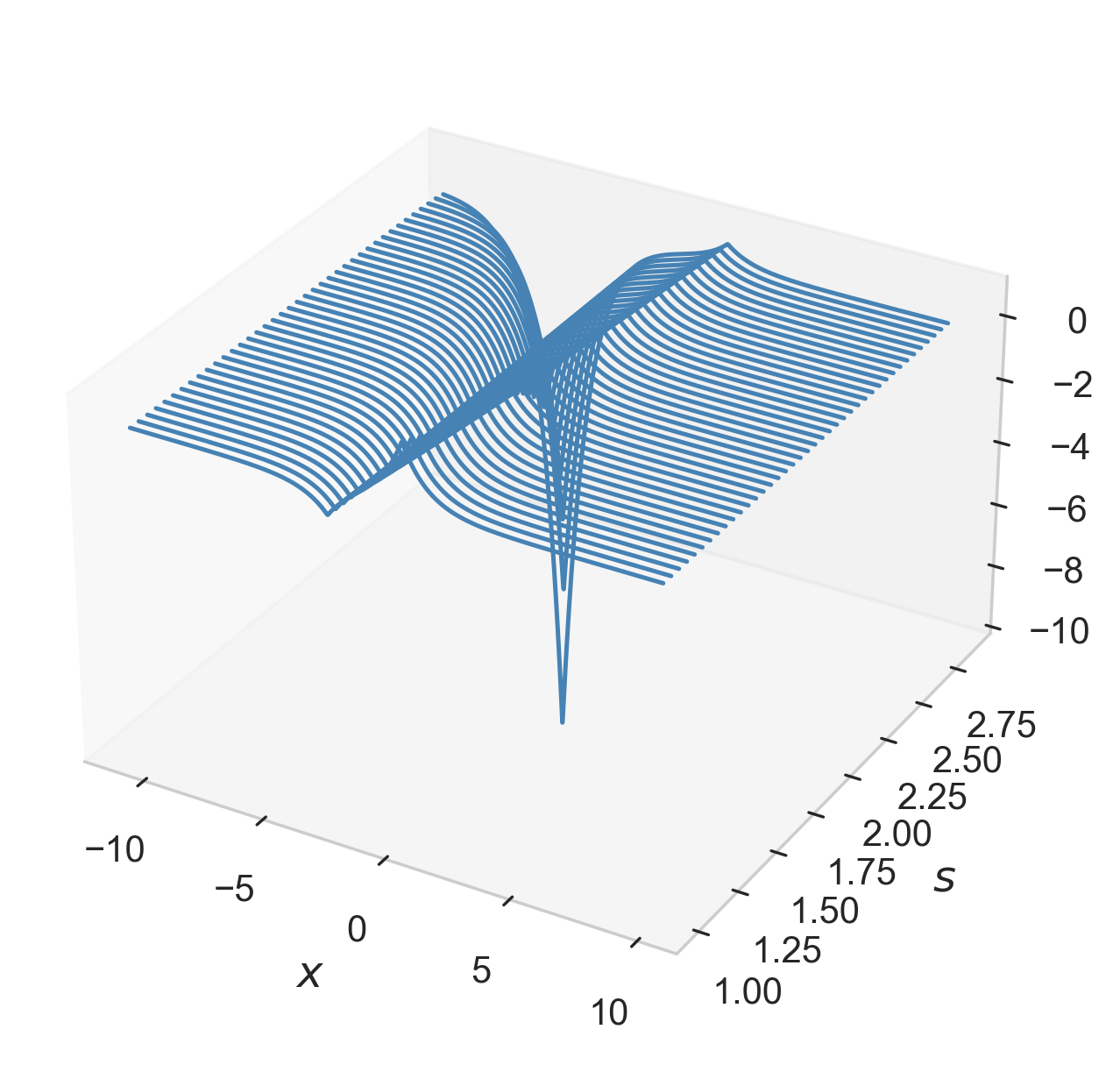}
            \textbf{(A)}
        \end{minipage}%
        \hfill
        \begin{minipage}[t]{0.48\textwidth}
            \centering
            \includegraphics[width=\textwidth]{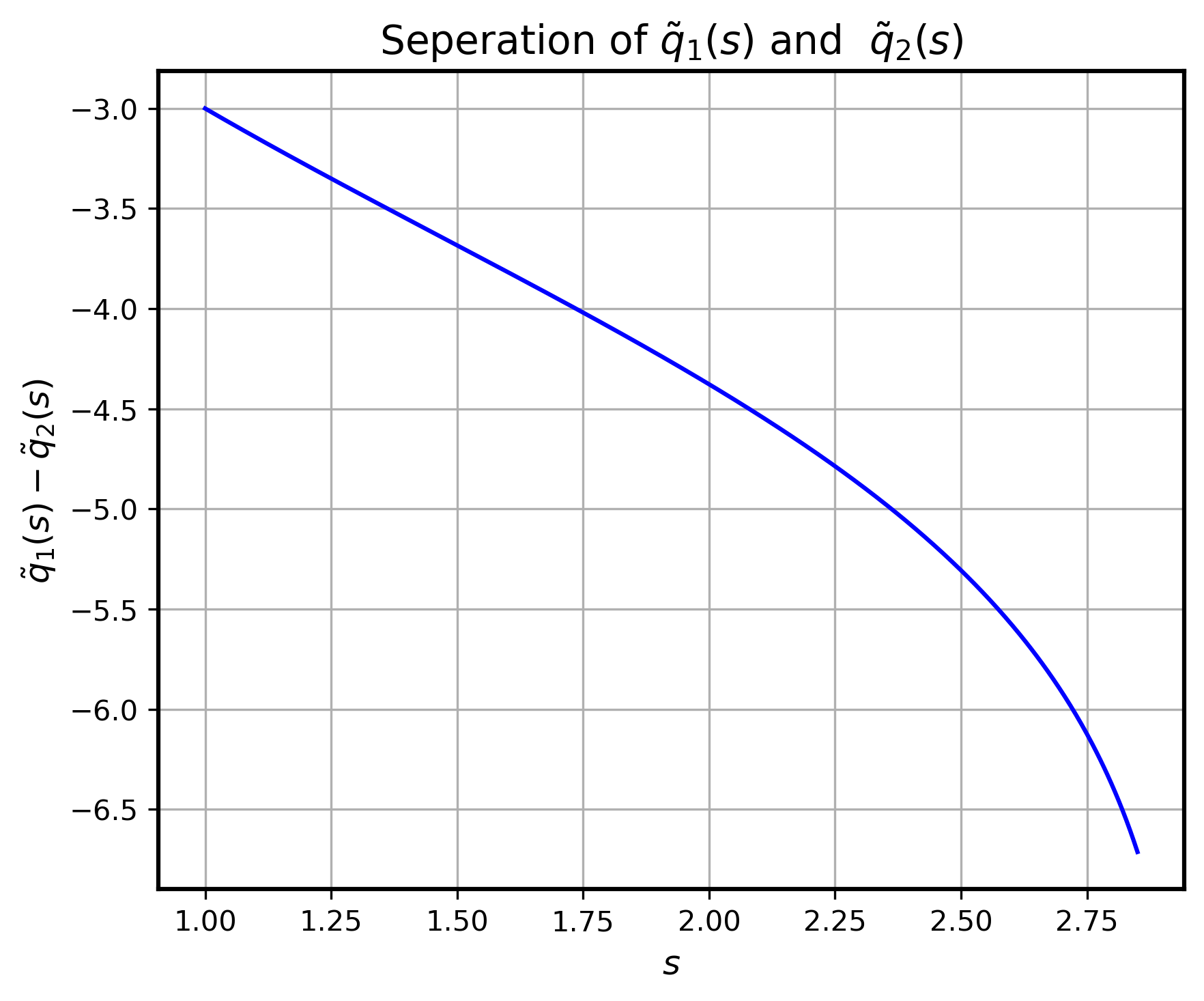}
            \textbf{(B)}
        \end{minipage}
        \caption{
            (A): Peakon-antipeakon Evolution with Initial Values \(\tilde  p_1(1) = -1 \), \( \tilde p_2(1) = 2 \), \( \tilde q_1(1) = -3 \), \(
          \tilde  q_2(1) = 0 \) and Parameters \(\mu_1 = 1\), \(\mu_2  = \frac{1}{3} \left(\frac{1}{e^3}-1\right)\),  \(\alpha =
            \frac{1}{1 -2e^3 }\), \( \beta = \frac{2e^3}{2e^3 - 1 }\), \(C_1 + C_2 = 2-\frac{1}{e^3} > 0\).
            (B): Relative Position of Two Peakons.}
            \label{fig:labelmp}
    \end{figure}

    \begin{figure}[H]
        \centering
        \begin{minipage}[t]{0.48\textwidth}
            \centering
            \includegraphics[width=\textwidth]{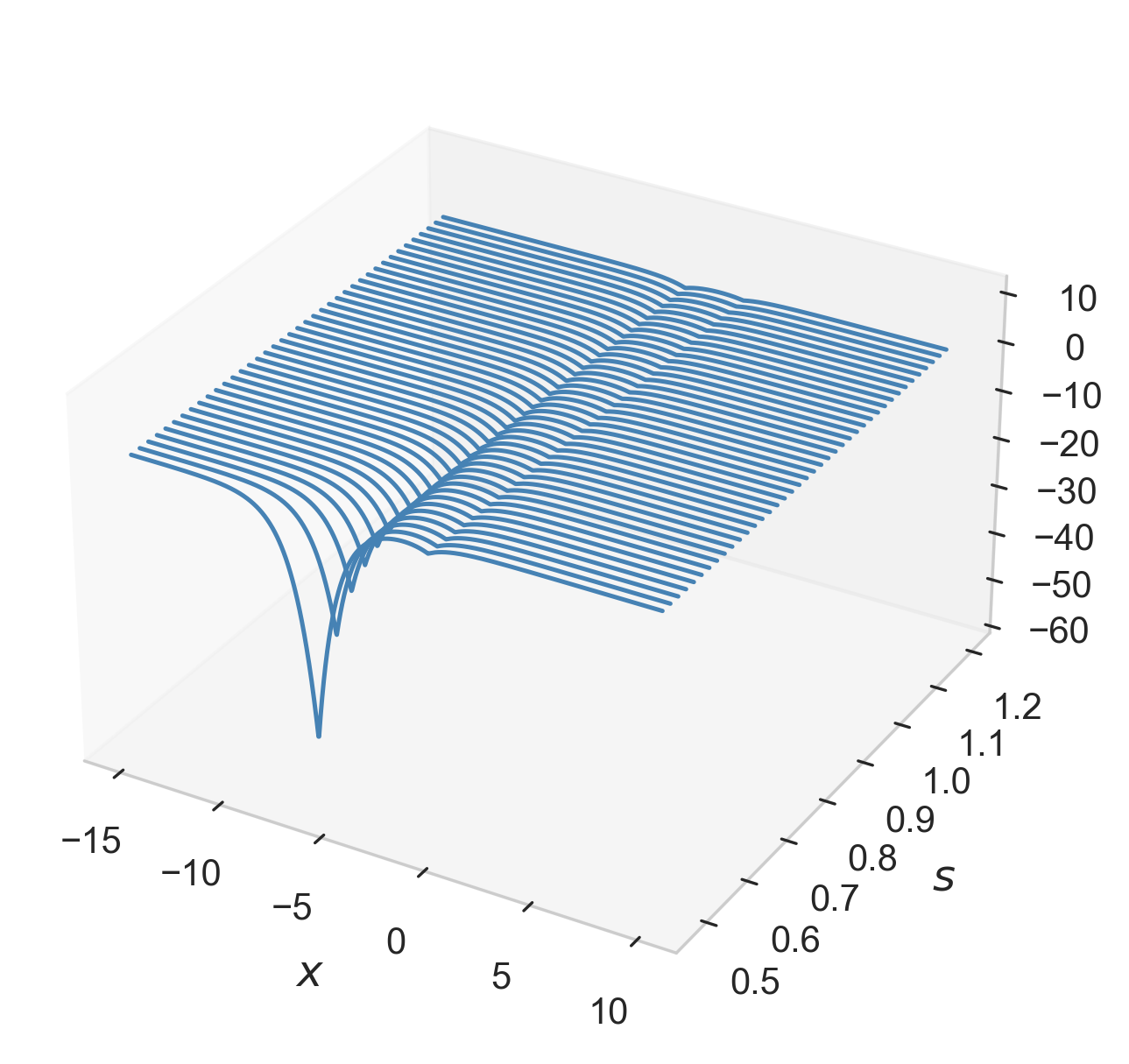}
            \textbf{(A)}
        \end{minipage}%
        \hfill
        \begin{minipage}[t]{0.48\textwidth}
            \centering
            \includegraphics[width=\textwidth]{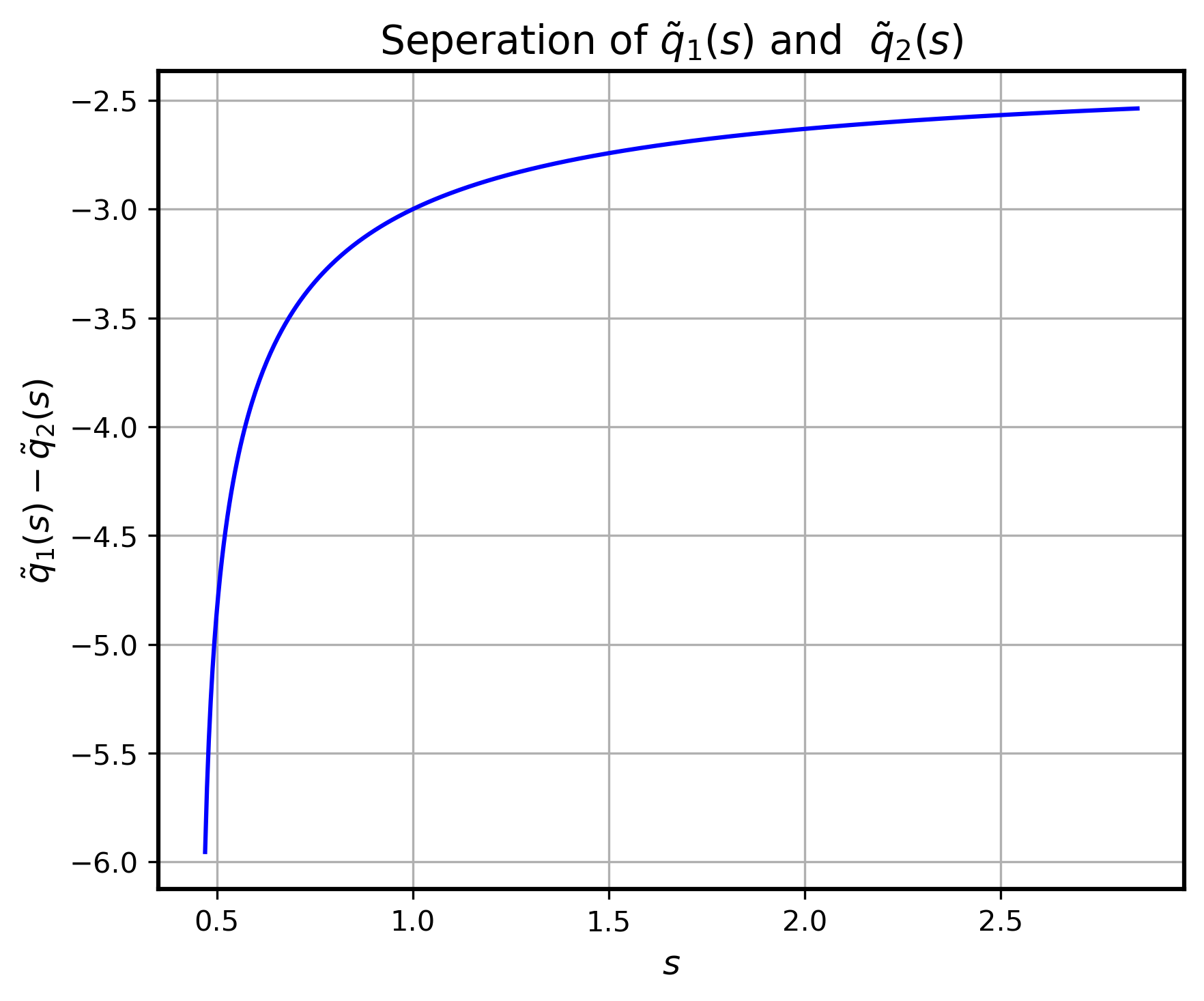}
            \textbf{(B)}
        \end{minipage}
        \caption{
            (A): Two-peakon evolution (Negative Amplitudes) with Initial Values \(  \tilde p_1(1) = -2 \), \( \tilde p_2(1) = -1 \), \(
          \tilde  q_1(1) = -3 \), \( \tilde q_2(1) = 0 \) and Parameters \(\mu_1 = 1\), \(\mu_2  =-2 \left(1-\frac{1}{e^3}\right)\),
            \(\alpha = \frac{2}{e^3 + 2 }\), \( \beta = \frac{e^3}{2 + e^3}\) ,\(C_1 + C_2 = -1-\frac{2}{e^3} < 0\).
            (B): Relative Position of Two Peakons.}
            \label{fig:labelmm}
    \end{figure}

\section{Analytical properties for the Cauchy problem}
In this section, we consider the Cauchy problem corresponding to equation \eqref{Nonlocallambda}, which is given by
\begin{equation}\label{Cauchyp}
\left\{
    \begin{aligned}
        &u_t + \lambda_1 u^2 + \lambda_2 u u_x + G\ast(\lambda_3 u^2 + \lambda_4 u_x^2) + \partial_x \left(  G \ast (\lambda_5 u^2 + \lambda_6 u_x^2) \right) = 0, \quad x\in\mathbb{R},\ t>0, \\
        &u(x,0)=u_0, \quad x\in\mathbb{R}.
    \end{aligned}
\right.
\end{equation}

Applying the Helmholtz operator \(1-\partial_x^2\), we can rewrite equation \eqref{Cauchyp} into an equivalent local form:
\begin{equation}\label{em}
m_t+(3\lambda_2-2\lambda_6)u_xm+\lambda_2um_x+2\lambda_1um+(\lambda_3-\lambda_1)u^2
+(2\lambda_6+2\lambda_5-3\lambda_2)uu_x+(\lambda_4 -2\lambda_1)u_x^2=0,
\end{equation}
where $m = u - u_{xx}$.

\subsection{Local well-posedness and blow-up scenario}
The main objective of this subsection is to establish the local well-posedness of the Cauchy problem \eqref{Cauchyp} in Besov spaces. To this end, we first introduce the following function spaces.
\begin{definition}\label{besov}\emph{\cite{BCD2011,Chemin1998}}
For $T>0$, $s\in\mathbb{R}$ and $1\leq p\leq +\infty$, we set
\begin{eqnarray*}
&&E_{p,r}^s(T)=:C([0, T];B_{p,r}^{s})\cap C^1([0, T];B_{p,r}^{s-1}),\quad if~~ r<\infty,\\
&&E_{p,\infty}^s(T)=:L^\infty([0, T];B_{p,\infty}^{s})\cap Lip([0, T];B_{p,\infty}^{s-1}),
\end{eqnarray*}
and \begin{eqnarray*} E_{p,r}^s=:\bigcap_{T>0}E_{p,r}^s(T).\end{eqnarray*}
\end{definition}

Our main result concerning the local existence is stated in the following theorem.
\begin{theorem}\label{LSS}
Suppose that $1\leq p,r\leq\infty$ and $s>\max \left\{1+\frac{1}{p},\frac{3}{2}\right\}$. If
$u_0\in B_{p,r}^{s}$, then there exists a time $T>0$ such that the Cauchy problem \eqref{Cauchyp} with initial data $u_0$ has a unique
solution $u\in E_{p,r}^s(T)$. The map $u_0\mapsto u$ is continuous from a neighborhood of $u_0$ in $B_{p,r}^{s}$ to
$$C([0, T];B_{p,r}^{s'})\cap C^1([0, T];B_{p,r}^{s'-1})$$
for every $s'<s$ when $r=+\infty$ and $s'=s$ when $r<+\infty$.
\end{theorem}
\begin{remark}
When $p=r=2$, the Besov space $B_{2,2}^{s}$ is equivalent to the Sobolev
space $H^s$. Theorem \ref{LSS} implies that
we can obtain the local well-posedness for the Cauchy problem \eqref{Cauchyp} in the Sobolev space under the condition $u_0\in H^s$ with $s>\frac{3}{2}$.
\end{remark}
Additionally, we can establish the following local well-posedness result in the critical Besov space $B_{2,1}^{3/2}$.
\begin{theorem}\label{LSS1}
Let $u_0\in B_{2,1}^{3/2}$, there exists a time $T>0$ such that the Cauchy problem (1.1) has a unique
solution $u\in E_{2,1}^{3/2}(T)$. The map $u_0\mapsto u$ is continuous from a neighborhood of $u_0$ in $B_{2,1}^{3/2}$ to
$C([0, T];B_{2,1}^{3/2})\cap C^1([0, T];B_{2,1}^{1/2}).$
\end{theorem}

To ensure the conciseness of this paper, we omit the detailed proofs of Theorems \ref{LSS} and \ref{LSS1}, as similar proofs can be found in the literature for the Camassa-Holm equations \cite{Dan2001,Dan2003} or for CH-type equations with rogue peakons \cite{zhu2023rogue}.

Before going to the main blow-up scenario \cite{chen2010blowup}, let us introduce two useful Kato-Ponce inequalities.
\begin{lemma}\emph{\cite{KP1988CPAM}}
If $r>0$, then $H^r\cap L^\infty$ is an algebra. Moreover
\begin{equation}\label{KP1}
\|fg\|_{H^r}\leq c(\|f\|_{L^\infty}\|g\|_{H^r}+\|f\|_{H^r}\|g\|_{L^\infty}),
\end{equation}
where $c$ is a constant depending only on $r$.
\end{lemma}
\begin{lemma}\emph{\cite{KP1988CPAM}}
If $r>0$, 
\begin{equation}\label{KP2}
\|[\Lambda^s,f]g\|_{L^2}\leq c(\|\partial_xf\|_{L^\infty}\|\Lambda^{r-1}g\|_{L^2}
+\|\Lambda^sf\|_{L^2}\|g\|_{L^\infty}),
\end{equation}
where $c$ is a constant depending only on $r$.
\end{lemma}
Now, we present the precise blow-up scenario for the Cauchy problem \eqref{Cauchyp}.
\begin{theorem}\label{bls1}
Let $u_0(x)\in H^s(\mathbb{R})$, $s>\frac{3}{2}$ and let $T$ be the existence time of the solution $u(x,t)$ to  the Cauchy problem \eqref{Cauchyp} with the initial data $u_0(x)$. Suppose that there exist $M > 0$ such that
\begin{eqnarray}\label{blsc}
\|u(\cdot,t)\|_{L^\infty}+\|u_x(\cdot,t)\|_{L^\infty}\leq M, \quad t\in [0,T),
\end{eqnarray}
then the $H^s$-norm of $u(x,t)$ does not blow up on $[0,T)$.
\end{theorem}
\begin{proof}
Let $\Lambda=(1-\partial_x^2)^{\tfrac{1}{2}}$. Applying $\Lambda^s$ on \eqref{Cauchyp}, multiplying by
$\Lambda^s u$ and integrating with respect to $x$ over $\mathbb{R}$, we have
\begin{equation}\label{Hsblow-up}
\begin{aligned}
\frac{1}{2}\frac{\mathrm{d}}{\mathrm{d}t}\|u\|_{H^s}^2
&=-\lambda_1\int_{\mathbb{R}}\Lambda^s( u^2)\Lambda^su \,\mathrm{d}x
-\lambda_2\int_{\mathbb{R}}\Lambda^s( u u_x)\Lambda^su \,\mathrm{d}x \\
&\quad-\int_{\mathbb{R}}\Lambda^s( G\ast(\lambda_3 u^2 + \lambda_4 u_x^2))\Lambda^su \,\mathrm{d}x
-\lambda_1\int_{\mathbb{R}}\Lambda^s( \partial_x \left(  G \ast (\lambda_5 u^2 + \lambda_6 u_x^2) \right))\Lambda^su \,\mathrm{d}x.
\end{aligned}
\end{equation}
Then, using the H\"{o}lder inequality, the Young inequality, the Kato-Ponce inequalities \eqref{KP1} and \eqref{KP2}, we estimate the four terms on the right hand side of \eqref{Hsblow-up} as follows:
\begin{equation}\label{Term1}
\begin{aligned}
&\quad\left|\lambda_1\int_{\mathbb{R}}\Lambda^s( u^2)\Lambda^su \,\mathrm{d}x\right|\leq c\|\Lambda^s(u^2)\|_{L^2}\|\Lambda^su\|_{L^2}\\
&\leq c\|u\|_{L^\infty}\|u\|_{H^s}^2,
\end{aligned}
\end{equation}
\begin{equation}\label{Term2}
\begin{aligned}
&\quad\left|\lambda_2\int_{\mathbb{R}}\Lambda^s( uu_x)\Lambda^su \,\mathrm{d}x\right|\\
&\leq\left|\lambda_2\int_{\mathbb{R}}\left(\Lambda^s( uu_x)- u \Lambda^s u_x\right)\Lambda^su \,\mathrm{d}x\right|+\left|\lambda_2\int_{\mathbb{R}} u \Lambda^s u_x\Lambda^su \,\mathrm{d}x\right|\\
&\leq c\|[\Lambda^s,u]u_x\|_{L^2} \|\Lambda^su\|_{L^2}+\left|\lambda_2\int_{\mathbb{R}} u_x (\Lambda^su)^2 \,\mathrm{d}x\right|\\
&\leq c\|\partial_x u\|_{L^\infty}\|u\|_{H^s}^2,
\end{aligned}
\end{equation}
\begin{equation}\label{Term3}
\begin{aligned}
&\quad\left|\int_{\mathbb{R}}\Lambda^s( G\ast(\lambda_3 u^2 + \lambda_4 u_x^2))\Lambda^su \,\mathrm{d}x\right|\\
&=\left|\int_{\mathbb{R}}\Lambda^{s-1}( \lambda_3 u^2 + \lambda_4 u_x^2)\Lambda^{s-1}u \,\mathrm{d}x\right|\\
&\leq c(\|u\|_{L^\infty}\|u\|_{H^{s-1}}+\|u_x\|_{L^\infty}\|u_x\|_{H^{s-1}})\|u\|_{H^{s-1}}\\
&\leq c(\|u\|_{L^\infty}+\|u_x\|_{L^\infty})\|u\|_{H^{s}}^2,
\end{aligned}
\end{equation}
and
\begin{equation}\label{Term4}
\begin{aligned}
&\quad\left|\int_{\mathbb{R}}\Lambda^s( \partial_x(G\ast(\lambda_5 u^2 + \lambda_6 u_x^2)))\Lambda^su \,\mathrm{d}x\right|\\
&=\left|\int_{\mathbb{R}}\Lambda^{s-1}( \lambda_5 u^2 + \lambda_6 u_x^2)\Lambda^{s}u \,\mathrm{d}x\right|\\
&\leq c(\|u\|_{L^\infty}\|u\|_{H^{s-1}}+\|u_x\|_{L^\infty}\|u_x\|_{H^{s-1}})\|u\|_{H^{s}}\\
&\leq c(\|u\|_{L^\infty}+\|u_x\|_{L^\infty})\|u\|_{H^{s}}^2.
\end{aligned}
\end{equation}
Combining \eqref{Term1}-\eqref{Term4} in \eqref{Hsblow-up}, we know that
\begin{equation*}
\frac{\mathrm{d}}{\mathrm{d}t}\|u\|_{H^s}^2\leq c(\|u\|_{L^\infty}+\|u_x\|_{L^\infty})\|u\|_{H^{s}}^2.
\end{equation*}
By Gronwall's inequality and the condition \eqref{blsc}, we get
\begin{equation*}
\|u\|_{H^s}^2\leq \mathrm{e}^{cMt}\|u_0\|_{H^s}^2, \quad t\in [0,T).
\end{equation*}
This completes the proof of Theorem \ref{bls1}.
\end{proof}

We give a precise blow-up scenario for a special case.
\begin{theorem}\label{bls2}
If the coefficients in \eqref{Cauchyp} satisfy the following relations:
\begin{equation}\label{constraintsforbu}
\lambda_3 = \lambda_1,\quad \lambda_4 = 2\lambda_1,\quad \lambda_5 = \frac32\lambda_2 - \lambda_6,
\end{equation}
let $u_0(x)\in H^s(\mathbb{R})$, $s>\tfrac{3}{2}$ and $T$ be the maximal existence time of the solution $u(x,t)$ to \eqref{Cauchyp} with the initial data
$u_0(x)$. Then the corresponding solution blows up in finite time if and only if
\begin{eqnarray*}
\liminf_{t\rightarrow T^-}\inf_{x\in\mathbb{R}}\left[2\lambda_1u+\left(\frac{5}{2}\lambda_2-2\lambda_6\right)u_x\right]=-\infty.
\end{eqnarray*}
\end{theorem}

\begin{proof}
We only give the proof for the space $H^2$. For general $ H^s(\mathbb{R})$, $s>\tfrac{3}{2}$, the reader can refer to the literature on CH \cite{constantin1998well} or DP \cite{liu2006global}. Under the relations \eqref{constraintsforbu}, we rewrite equation \eqref{em} as
\begin{equation}\label{em1}
m_t+(3\lambda_2-2\lambda_6)u_xm+\lambda_2um_x+2\lambda_1um=0.
\end{equation}
On the one hand, suppose for any $t\in (0, T]$,  $\inf\limits_{x\in\mathbb{R}}\left[\left(\frac{5}{2}\lambda_2-2\lambda_6\right)u_x+2\lambda_1u\right]\geq -M_0$ with $M_0>0$.
Multiplying \eqref{em1} by $m$, after integrating by parts, we have
\begin{equation*}
\begin{aligned}
\frac{1}{2}\frac{\mathrm{d}}{\mathrm{d}t}\|m\|_{L^2}^2&=-\int_{\mathbb{R}}(3\lambda_2-2\lambda_6)u_xm^2+\lambda_2umm_x+2\lambda_1um^2\,\mathrm{d}x\\
&=-\int_{\mathbb{R}}\left[\left(\frac{5}{2}\lambda_2-2\lambda_6\right)u_x+2\lambda_1u\right]m^2\,\mathrm{d}x.
\end{aligned}
\end{equation*}
By using Gronwall's inequality, we have
\begin{eqnarray*}
\|m(t)\|_{L^2}^2\leq\|m_0\|_{L^2}^2e^{2M_0T},
\end{eqnarray*}
then $\|u(t)\|_{H^2}^2\leq\|m(t)\|_{L^2}^2$ is bounded. This contradicts the fact that $T$ is the maximal time of existence.

On the other hand, the solution does not blow up in $H^s, s>\tfrac{3}{2}$, that is $\|u\|_{H^s}$ is bounded, by Morrey's inequality, we have
\begin{eqnarray*}
\left\|\left(\frac{5}{2}\lambda_2-2\lambda_6\right)u_x+2\lambda_1u\right\|_{L^\infty}\leq C\|u\|_{H^s}<+\infty.
\end{eqnarray*}
This completes the proof of Theorem \ref{bls2}.
\end{proof}
\subsection{Global existence}
In this subsection, we establish sufficient conditions on the initial data to ensure the global existence of the solution.
\begin{lemma}\label{CLH1}
If the coefficients in \eqref{Cauchyp} satisfy the following relations:
\begin{equation*}\label{constraintsforH1}
\lambda_1+\lambda_3=0,\quad 2\lambda_1+\lambda_4 =0,\quad \lambda_2 = 2\lambda_6,
\end{equation*}
then, equation \eqref{Cauchyp} has the conservation law
\begin{equation*}
\left\| u\right\| _{H^1}=\left\| u_{0}\right\|_{H^1} .
\end{equation*}
\end{lemma}
\begin{proof}
Multiplying \eqref{Cauchyp} by $m$ and integrating with respect to $x$ over $\mathbb{R}$, we have
\begin{equation*}
    \begin{aligned}
 \frac{1}{2}\frac{d}{dt}\|u\|_{H^1} &=-\int_{\mathbb{R}}\lambda_1 u^2m + \lambda_2 u u_xm + G\ast(\lambda_3 u^2 + \lambda_4 u_x^2)m + \partial_x \left(  G \ast (\lambda_5 u^2 + \lambda_6 u_x^2)\right)m dx\\
 &=-\int_{\mathbb{R}}\lambda_1 u^2m + \lambda_2 u u_xm + (\lambda_3 u^2 + \lambda_4 u_x^2)u + \partial_x  (\lambda_5 u^2 + \lambda_6 u_x^2)u dx\\
 &=-\int_{\mathbb{R}}(\lambda_1+\lambda_3) u^3 + (2\lambda_1+\lambda_4) u u_x^2
 +\frac{\lambda_2-\lambda_5}{3}\partial_x(u^3) + \left(\frac{\lambda_2}{2}-\lambda_6\right)u_x^3dx\\
 &=0.
 \end{aligned}
\end{equation*}
This completes the proof of Lemma \ref{CLH1}.
\end{proof}
Before going to our main result, for $\lambda_2\neq0$, we introduce the characteristics $q(x,t)$ associated to the equation \eqref{Cauchyp} as
\begin{equation}\label{characteristics}
\left\{\begin{array}{l}
q_{t}(x,t)=\lambda_2u(q(x,t),t),\quad x \in {\mathbb{R}}, \quad t \in[0, T), \\
q(x,0)=x.
\end{array}\right.
\end{equation}
where $T$ is the life-span of the solution.
Taking the derivative \eqref{characteristics} with respect to $x$, we obtain
$$\frac{dq_t}{dx}=q_{tx}=\lambda_2u_x(q,t)q_x, \qquad t\in(0,T).$$
Therefore,
\begin{eqnarray*}\label{qx}
\left\{
  \begin{array}{ll}
   q_x=\exp\left(\int^t_0\lambda_2u_x(q,s)ds\right), \qquad 0<t<T, x\in \mathbb{R},\\
   q_x(x,0)=1, \qquad x\in \mathbb{R},
  \end{array}
\right.
\end{eqnarray*}
which is always positive before the blow-up time. Therefore, the function $q(x,t)$ is an increasing diffeomorphism of the line before it blows up.
\begin{theorem}\label{GE}
If the coefficients in \eqref{Cauchyp} satisfy the following relations:
\begin{equation}\label{constraintsforGE}
\lambda_1+\lambda_3=0,\quad 2\lambda_1+\lambda_4 =0,\quad \lambda_2 = \lambda_6=0.
\end{equation}
Suppose that $u_0(x)\in H^s(\mathbb{R})$, $s>\frac{3}{2}$, then, the corresponding solution $u(x, t)$ to
the Cauchy problem \eqref{Cauchyp} with $u_0$ as the initial datum exists globally.
\end{theorem}
\begin{proof}
Under the relations \eqref{constraintsforGE}, the $H^1$-norm of solution $u(x,t)$ is conserved. 
By inequality $\|u\|_{L^\infty}\leq \frac{1}{\sqrt{2}}\|u\|_{H^1}$
and blow-up scenario (Theorem \ref{bls1}), we only need to prove that for any time $T>0$, there exists a number $M(T)>0$, such that $\|u_x(\cdot,T)\|_{L^\infty}\leq M(T)$.

Under the relations \eqref{constraintsforGE}, equation \eqref{Cauchyp} reduces to
\begin{equation*}
u_t + \lambda_1 u^2  + G\ast(\lambda_3 u^2 + \lambda_4 u_x^2) + \partial_x \left(  G \ast (\lambda_5 u^2 ) \right) = 0.
\end{equation*}
Differentiating the above equation with respect to $x$, we have
\begin{equation*}
u_{xt} + 2\lambda_1 uu_x  +\partial_xG\ast(\lambda_3 u^2 + \lambda_4 u_x^2) +  G \ast (\lambda_5 u^2 ) -\lambda_5 u^2 = 0.
\end{equation*}
By fundamental ODE methods, for any $T>0$, we have
\begin{equation*}
u_{x}(x,T) =e^{\int_0^T-2\lambda_1 udt}\left(u_{0x}(x)
-\int_0^T\left(\partial_xG\ast(\lambda_3 u^2 + \lambda_4 u_x^2) +  G \ast (\lambda_5 u^2 ) -\lambda_5 u^2 \right)
e^{\int_0^s2\lambda_1 u(x,\tau)d\tau}ds\right).
\end{equation*}
It follows that for $s>\tfrac{3}{2}$
\begin{equation*}
\begin{aligned}
|u_{x}(x,T)| &\leq e^{|\lambda_1| \|u_0\|_{H^1}T}\left(u_{0x}(x)+
\int_{0}^{T}\left(\frac{1}{2}|\lambda_3|+\frac{1}{2}|\lambda_4|+|\lambda_5|\right)
\|u_0\|_{H^1}^2e^{|\lambda_1| \|u_0\|_{H^1}s}ds\right)\\
& \leq e^{|\lambda_1| \|u_0\|_{H^1}T}\left(u_{0x}(x)+
\left(\frac{1}{2}|\lambda_3|+\frac{1}{2}|\lambda_4|+|\lambda_5|\right)
\|u_0\|_{H^1}^2e^{|\lambda_1| \|u_0\|_{H^1}T}T\right).
\end{aligned}
\end{equation*}
This means $\|u_x(\cdot,T)\|_{L^\infty}\leq M(T)$.

This completes the proof of Theorem \ref{GE}.
\end{proof}

\begin{theorem}\label{GE1}
If the coefficients in \eqref{Cauchyp} satisfy the following relations:
\begin{equation}\label{constraintsforGE1}
\lambda_1=\lambda_3,\quad\lambda_4 =2\lambda_1,\quad \lambda_5 = \frac32\lambda_2 - \lambda_6,\quad \lambda_2\neq0.
\end{equation}
Suppose that $u_0(x)\in H^s(\mathbb{R})$, $s\geq 2$, 
\begin{eqnarray}
&m_0(x)\geq(\not\equiv) 0,\quad \lambda_1\geq\left|\frac{1}{2}\left(\frac{5}{2}\lambda_2-2\lambda_6\right)\right|;\label{constraints2forGE1}\\
\mathrm{or}\quad &m_0(x)\leq(\not\equiv) 0,\quad \lambda_1\leq-\left|\frac{1}{2}\left(\frac{5}{2}\lambda_2-2\lambda_6\right)\right|. \label{constraints2forGE2}
\end{eqnarray}
Then, the corresponding solution $u(x, t)$ to
the Cauchy problem \eqref{Cauchyp} with $u_0$ as the initial datum exists globally.
\end{theorem}
\begin{proof}
Under the relations \eqref{constraintsforGE1}, equation \eqref{em} reduces to 
\begin{equation*}\label{em1}
m_t+(3\lambda_2-2\lambda_6)u_xm+\lambda_2um_x+2\lambda_1um=0.
\end{equation*} 
The operator $(1-\partial_x^2)^{-1} $ can be expressed by its associated Green's function as
$$u(x,t)=(1-\partial_x^2)^{-1} m(x,t)=G*m,\quad G(x)=\frac{1}{2}e^{-|x|}.$$
More precisely,
\begin{eqnarray*}\label{u}
u(x,t)=G\ast m(x,t)=\frac{1}{2}e^{-x}\int_{-\infty}^{x}e^{\xi}m(\xi,t)d\xi+\frac{1}{2}e^{x}\int_x^{\infty}e^{-\xi}m(\xi,t)d\xi.
\end{eqnarray*}
\begin{eqnarray*}\label{ux}
u_x(x,t)=-\frac{1}{2}e^{-x}\int_{-\infty}^{x}e^{\xi}m(\xi,t)d\xi+\frac{1}{2}e^{x}\int_x^{\infty}e^{-\xi}m(\xi,t)d\xi.
\end{eqnarray*}
From the blow-up scenario (Theorem \ref{bls2}), we need to show that $\left(\frac{5}{2}\lambda_2-2\lambda_6\right)u_x+2\lambda_1u$ is bounded from below.
Recall the definition of characteristics \eqref{characteristics}, for $\lambda_2\neq0$, we have 
\begin{equation*}
\begin{aligned}
\quad\frac{d}{dt}\left(e^{\tfrac{2\lambda_1}{\lambda_2}q}mq_x^{\tfrac{3\lambda_2-2\lambda_6}{\lambda_2}}\right)=\left((3\lambda_2-2\lambda_6)u_xm+\lambda_2um_x+2\lambda_1um\right)
e^{\tfrac{2\lambda_1}{\lambda_2}q}q_x^{\tfrac{3\lambda_2-2\lambda_6}{\lambda_2}}
=0.
\end{aligned}
\end{equation*}
Hence, the following identity can be proved:
\begin{equation}\label{IE}
e^{\tfrac{2\lambda_1}{\lambda_2}q}mq_x^{\tfrac{3\lambda_2-2\lambda_6}{\lambda_2}}
=e^{\tfrac{2\lambda_1}{\lambda_2}x}m_0,
\end{equation}
which implies that $m(x, t)$ keeps the sign with respect to the initial datum.
If the initial data condition $m_0(x)\geq(\not\equiv) 0$, by \eqref{IE}, we obtain
\begin{eqnarray}\label{mp1}
m(x,t)\geq(\not\equiv) 0.
\end{eqnarray}
If the initial data condition $m_0(x)\leq(\not\equiv) 0$, by \eqref{IE}, we obtain
\begin{eqnarray}\label{mp2}
m(x,t)\leq(\not\equiv) 0.
\end{eqnarray}
From \eqref{constraints2forGE1}, \eqref{constraints2forGE2}, \eqref{mp1} and \eqref{mp2}, we have
\begin{equation*}
\begin{aligned}
&\quad\left(\frac{5}{2}\lambda_2-2\lambda_6\right)u_x+2\lambda_1u\\
&=\left[\lambda_1-\frac{1}{2}\left(\frac{5}{2}\lambda_2-2\lambda_6\right)\right]
e^{-x}\int_{-\infty}^x e^\xi mdx
+\left[\lambda_1+\frac{1}{2}\left(\frac{5}{2}\lambda_2-2\lambda_6\right)\right]
e^{x}\int_{x}^{\infty}  e^{-\xi} mdx\\
&\geq0.
\end{aligned}
\end{equation*}
The proof of Theorem \ref{GE1} is completed.
\end{proof}

\subsection{Wave breaking phenomenon}
In this subsection, we study the wave breaking phenomenon. Wave breaking phenomenon means that the solution remains bounded but its slope becomes unbounded inﬁnite time. We introduce two Lemmas for ODE theory which will been used later.
\begin{lemma}\emph{\cite{Chen2016JFA}}\label{ybl1}
Assume that a differentiable function $\tilde{y}(t)$ satisfies
\begin{equation}\label{RTE2}
y'\geq Cy^2-K
\end{equation}
with constants $C, K > 0$. If the initial datum $y(0) = y_0 >\sqrt{\frac{K}{C}}$, then the solution $y(t)$ to \eqref{RTE2} goes to $+\infty$ before $t$ tend to 
$\frac{1}{2\sqrt{CK}}\log\left(\frac{y_0+\sqrt{\frac{K}{C}}}{y_0-\sqrt{\frac{K}{C}}}\right)$.
\end{lemma}
Let $\tilde{y}=-y$ in Lemma \ref{ybl1}, it is easy to get
\begin{lemma}\label{ybl2}
Assume that a differentiable function $y(t)$ satisfies
\begin{equation}\label{RTE1}
\tilde{y}'\leq-C\tilde{y}^2+K
\end{equation}
with constants $C, K > 0$. If the initial datum $\tilde{y}(0) = \tilde{y}_0 < -\sqrt{\frac{K}{C}}$, then the solution to \eqref{RTE1} goes to $-\infty$ before $t$ tend to 
$\frac{1}{2\sqrt{CK}}\log\left(\frac{\tilde{y}_0-\sqrt{\frac{K}{C}}}{\tilde{y}_0+\sqrt{\frac{K}{C}}}\right)$.
\end{lemma}

\begin{theorem}\label{WBP}
If the coefficients in \eqref{Cauchyp} satisfy the following relations:
\begin{equation*}
\lambda_1+\lambda_3=0,\quad 2\lambda_1+\lambda_4 =0,\quad \lambda_2 = 2\lambda_6.
\end{equation*}
Suppose that $u_0(x)\in H^s(\mathbb{R})$, $s>\frac{3}{2}$, if there exists a point $x_0$, such that 
\begin{equation}\label{BLC}
\left\{
  \begin{array}{ll}
 \left(u_{0x} -\frac{\lambda_1}{\lambda_6-\lambda_2}u_0\right)(x_0)
\leq -\sqrt{\frac{K_0}{\lambda_2-\lambda_6}}, 
&\lambda_6<\lambda_2, \\
 \left(u_{0x} -\frac{\lambda_1}{\lambda_6-\lambda_2}u_0\right)(x_0)
\geq \sqrt{\frac{K_0}{\lambda_6-\lambda_2}}, 
&\lambda_6>\lambda_2,
  \end{array}
\right.
\end{equation}
where 
\begin{equation}\label{KCC}
 \begin{aligned}
&K_0:=\left(\frac{|\lambda_5|}{2}+\max\{C_1,C_2\}\right)\|u_0\|_{H^1}^2,\\
&C_1:=\frac{1}{2}\left(|\lambda_3|+|\lambda_5|
+\left|\frac{\lambda_1\lambda_3}{\lambda_6-\lambda_2}\right|
+\left|\frac{\lambda_1\lambda_5}{\lambda_6-\lambda_2}\right|\right),\\
&C_2:=\frac{1}{2}\left(|\lambda_4|+|\lambda_6|
+\left|\frac{\lambda_1\lambda_4}{\lambda_6-\lambda_2}\right|
+\left|\frac{\lambda_1\lambda_6}{\lambda_6-\lambda_2}\right|\right).
\end{aligned}
\end{equation}
Then, the corresponding solution $u(x,t)$ blows up in finite time with an
estimate of the blow-up time $T^*$ as
\begin{equation*}
T^*:=\left\{
\begin{array}{ll}
T_1\leq \frac{1}{2\sqrt{(\lambda_2-\lambda_6)K_0}}
\log\left( \frac{\left(u_{0x} -\frac{\lambda_1}{\lambda_6-\lambda_2}u_0\right)(x_0)
-\sqrt{\frac{K_0}{\lambda_2-\lambda_6}}}
{\left(u_{0x} -\frac{\lambda_1}{\lambda_6-\lambda_2}u_{0}\right)(x_0)
+\sqrt{\frac{K_0}{\lambda_2-\lambda_6}}}\right), & \lambda_6<\lambda_2, \\
T_2\leq \frac{1}{2\sqrt{(\lambda_6-\lambda_2)K_0}}
\log\left( \frac{\left(u_{0x} -\frac{\lambda_1}{\lambda_6-\lambda_2}u_0\right)(x_0)
+\sqrt{\frac{K_0}{\lambda_6-\lambda_2}}}
{\left(u_{0x} -\frac{\lambda_1}{\lambda_6-\lambda_2}u_{0}\right)(x_0)
-\sqrt{\frac{K_0}{\lambda_6-\lambda_2}}}\right), & \lambda_6>\lambda_2.
       \end{array}
     \right.
\end{equation*}
\end{theorem}

\begin{remark}
In this Theorem, we needs a condition $\lambda_{2}\neq\lambda_6$. Actually, when $\lambda_2=\lambda_6$, equation \eqref{em} reduces to a second-order partial differential equation. It may exists locally or globally in a more low regularity Sobolev space. Such as Theorem \ref{GE}, we have established sufficient conditions for the solution to exist globally.
\end{remark}

\begin{proof}
Finding the partial derivative of the equation \eqref{Cauchyp} with respect to $x$ yields the following the equation
\begin{equation*}\label{u_x}
u_{xt} + 2\lambda_1 uu_x + \lambda_2 u u_{xx} +(\lambda_2-\lambda_6) u_x^2 +\partial_xG\ast(\lambda_3 u^2 + \lambda_4 u_x^2) +  G \ast (\lambda_5 u^2 + \lambda_6 u_x^2) -\lambda_5 u^2 = 0.
\end{equation*}
It follows that 
\begin{equation*}\label{u_x1}
 \begin{aligned}
\quad u_{xt} + \lambda_2 u u_{xx}
&= (\lambda_6-\lambda_2) \left(u_x -\frac{\lambda_1}{\lambda_6-\lambda_2}u\right)^2 
+\left(\lambda_5-\frac{\lambda_1^2}{\lambda_6-\lambda_2}\right) u^2\\
&\quad-\partial_xG\ast(\lambda_3 u^2 + \lambda_4 u_x^2) -   G \ast (\lambda_5 u^2 + \lambda_6 u_x^2).
 \end{aligned}
\end{equation*}
Then, we give the equation for $u_x -\frac{\lambda_1}{\lambda_6-\lambda_2}u$.
\begin{equation*}\label{u_x-u}
\begin{aligned}
&\quad\left(u_x -\frac{\lambda_1}{\lambda_6-\lambda_2}u\right)_t 
+ \lambda_2 u \left(u_x -\frac{\lambda_1}{\lambda_6-\lambda_2}u\right)_{x}\\
&= (\lambda_6-\lambda_2) \left(u_x -\frac{\lambda_1}{\lambda_6-\lambda_2}u\right)^2 
-\partial_xG\ast(\lambda_3 u^2 + \lambda_4 u_x^2) -   G \ast (\lambda_5 u^2 + \lambda_6 u_x^2)
+\lambda_5 u^2 \\
&\quad+\frac{\lambda_1}{\lambda_6-\lambda_2}G\ast(\lambda_3 u^2 +\lambda_4 u_x^2) + \frac{\lambda_1}{\lambda_6-\lambda_2}\partial_x \left(  G \ast (\lambda_5 u^2 + \lambda_6 u_x^2) \right) .
\end{aligned}
\end{equation*}
At the point $(q(x_0,t),t)$, we have 
\begin{equation*}\label{u_x-u}
\begin{aligned}
&\quad\left(u_x -\frac{\lambda_1}{\lambda_6-\lambda_2}u\right)_t(q(x_0,t),t) \\
&= (\lambda_6-\lambda_2) \left(u_x -\frac{\lambda_1}{\lambda_6-\lambda_2}u\right)^2 (q(x_0,t),t)
-\partial_xG\ast(\lambda_3 u^2 + \lambda_4 u_x^2) -   G \ast (\lambda_5 u^2 + \lambda_6 u_x^2)\\
&\quad+\lambda_5 u^2(q(x_0,t),t)+\frac{\lambda_1}{\lambda_6-\lambda_2}G\ast(\lambda_3 u^2 + \lambda_4 u_x^2) + \frac{\lambda_1}{\lambda_6-\lambda_2}\partial_x \left(  G \ast (\lambda_5 u^2 + \lambda_6 u_x^2) \right)\\
&= (\lambda_6-\lambda_2) \left(u_x -\frac{\lambda_1}{\lambda_6-\lambda_2}u\right)^2 (q(x_0,t),t)+\lambda_5 u^2(q(x_0,t),t)\\
&\quad
-\partial_xG\ast(\lambda_3 u^2 + \lambda_4 u_x^2) -   G \ast (\lambda_5 u^2 + \lambda_6 u_x^2)\\
&\quad+\frac{\lambda_1}{\lambda_6-\lambda_2}G\ast(\lambda_3 u^2 + \lambda_4 u_x^2) + \frac{\lambda_1}{\lambda_6-\lambda_2}\partial_x \left(  G \ast (\lambda_5 u^2 + \lambda_6 u_x^2) \right).
\end{aligned}
\end{equation*}
If $\lambda_6<\lambda_2$, 
\begin{equation*}
\begin{aligned}
&\quad\left(u_x -\frac{\lambda_1}{\lambda_6-\lambda_2}u\right)_t(q(x_0,t),t) \\
&\leq (\lambda_6-\lambda_2) \left(u_x -\frac{\lambda_1}{\lambda_6-\lambda_2}u\right)^2 (q(x_0,t),t)
+\frac{|\lambda_5|}{2}\|u_0\|_{H^1}^2\\
&\quad+\frac{1}{2}\left(|\lambda_3|+|\lambda_5|
+\left|\frac{\lambda_1\lambda_3}{\lambda_6-\lambda_2}\right|
+\left|\frac{\lambda_1\lambda_5}{\lambda_6-\lambda_2}\right|\right)\|u\|_{L^2}\\
&\quad+\frac{1}{2}\left(|\lambda_4|+|\lambda_6|
+\left|\frac{\lambda_1\lambda_4}{\lambda_6-\lambda_2}\right|
+\left|\frac{\lambda_1\lambda_6}{\lambda_6-\lambda_2}\right|\right)\|u_x\|_{L^2}\\
&\leq (\lambda_6-\lambda_2) \left(u_x -\frac{\lambda_1}{\lambda_6-\lambda_2}u\right)^2 (q(x_0,t),t)
+\left(\frac{|\lambda_5|}{2}+\max\{C_1,C_2\}\right)\|u_0\|_{H^1}^2\\
&\leq (\lambda_6-\lambda_2) \left(u_x -\frac{\lambda_1}{\lambda_6-\lambda_2}u\right)^2 (q(x_0,t),t)
+K_0,
\end{aligned}
\end{equation*}
where $K_0,C_1,C_2$ are defined in \eqref{KCC}.

By applying Lemma \ref{bls2} and \eqref{BLC}, we have 
\begin{equation*}
\lim_{t\rightarrow T_1}\left(u_x -\frac{\lambda_1}{\lambda_6-\lambda_2}u\right)(q(x_0,t),t)=-\infty,
\end{equation*}
with 
\begin{equation*}
T_1\leq \frac{1}{2\sqrt{(\lambda_2-\lambda_6)K_0}}
\log\left( \frac{\left(u_{0x} -\frac{\lambda_1}{\lambda_6-\lambda_2}u_0\right)(x_0)
-\sqrt{\frac{K_0}{\lambda_2-\lambda_6}}}
{\left(u_{0x} -\frac{\lambda_1}{\lambda_6-\lambda_2}u_{0}\right)(x_0)
+\sqrt{\frac{K_0}{\lambda_2-\lambda_6}}}\right).
\end{equation*}
If $\lambda_6>\lambda_2$, 
\begin{equation*}
\begin{aligned}
&\quad\left(u_x -\frac{\lambda_1}{\lambda_6-\lambda_2}u\right)_t(q(x_0,t),t) \\
&\geq (\lambda_6-\lambda_2) \left(u_x -\frac{\lambda_1}{\lambda_6-\lambda_2}u\right)^2 (q(x_0,t),t)
-\frac{|\lambda_5|}{2}\|u_0\|_{H^1}^2\\
&\quad-\frac{1}{2}\left(|\lambda_3|+|\lambda_5|
+\left|\frac{\lambda_1\lambda_3}{\lambda_6-\lambda_2}\right|
+\left|\frac{\lambda_1\lambda_5}{\lambda_6-\lambda_2}\right|\right)\|u\|_{L^2}\\
&\quad-\frac{1}{2}\left(|\lambda_4|+|\lambda_6|
+\left|\frac{\lambda_1\lambda_4}{\lambda_6-\lambda_2}\right|
+\left|\frac{\lambda_1\lambda_6}{\lambda_6-\lambda_2}\right|\right)\|u_x\|_{L^2}\\
&\geq (\lambda_6-\lambda_2) \left(u_x -\frac{\lambda_1}{\lambda_6-\lambda_2}u\right)^2 (q(x_0,t),t)
-\left(\frac{|\lambda_5|}{2}+\max\{C_1,C_2\}\right)\|u_0\|_{H^1}^2\\
&\geq (\lambda_6-\lambda_2) \left(u_x -\frac{\lambda_1}{\lambda_6-\lambda_2}u\right)^2 (q(x_0,t),t)
-K_0.
\end{aligned}
\end{equation*}
By applying Lemma \ref{bls1} and \eqref{BLC}, we have 
\begin{equation*}
\lim_{t\rightarrow T_2}\left(u_x -\frac{\lambda_1}{\lambda_6-\lambda_2}u\right)(q(x_0,t),t)=\infty,
\end{equation*}
with 

\begin{equation*}
T_2\leq \frac{1}{2\sqrt{(\lambda_6-\lambda_2)K_0}}
\log\left( \frac{\left(u_{0x} -\frac{\lambda_1}{\lambda_6-\lambda_2}u_0\right)(x_0)
+\sqrt{\frac{K_0}{\lambda_6-\lambda_2}}}
{\left(u_{0x} -\frac{\lambda_1}{\lambda_6-\lambda_2}u_{0}\right)(x_0)
-\sqrt{\frac{K_0}{\lambda_6-\lambda_2}}}\right).
\end{equation*}
This completes the proof of Theorem \ref{WBP}.

\end{proof}

\subsection{Ill-posedness in Besov space $B_{2,\infty}^{3/2}$}
In this subsection, we investigate the ill-posedness in Besov space $B_{2,\infty}^{3/2}$.
\begin{theorem}\label{IPB}
If the coefficients in \eqref{Cauchyp} satisfy the following relations:
\begin{equation*}\label{constraintsforill}
\lambda_3 = \lambda_1,\quad \lambda_4 = 2\lambda_1,\quad \lambda_5 = \frac32\lambda_2 - \lambda_6,
\quad \lambda_2\neq0,
\end{equation*}
then the Cauchy problem \eqref{Cauchyp} is ill-posedness for $u(0)\in{B_{2,\infty}^{3/2}}$. More precisely, there
is a solution $u\in B_{2,\infty}^{3/2}$ of \eqref{Cauchyp} with the initial data $u(0)$ such that for any $T$, $\varepsilon>0$, there exists a solution $v\in L^\infty([0,T];{B_{2,\infty}^{3/2}})$ with
\begin{eqnarray*}
\|v(0)-u(0)\|_{B_{2,\infty}^{3/2}}\leq \varepsilon \quad \emph{and} \quad \|v(t)-u(t)\|_{L([0,T];B_{2,\infty}^{3/2})}\geq \frac{1}{2\lambda_1^2}.
\end{eqnarray*}
\end{theorem}
\begin{proof}
The proof of Theorem \ref{IPB} is divided into two cases: \(\lambda_1 = 0\) and \(\lambda_1 \neq 0\).

For the case \(\lambda_1 = 0\), \(\lambda_2 \neq0\), the Cauchy problem \eqref{Cauchyp} admits the single peakon solution as
\begin{equation*}
u(x,t) = c \mathrm{e}^{-|x - \lambda_2ct|},
\end{equation*}
where \(c\) is an arbitrary constant.
Since we can follow the work step by step in \cite{Dan2003}, we omit the details. 

Our main task is to discuss the second case \(\lambda_1\neq 0\), \(\lambda_2 \neq0\).  Recall the non-traveling wave solutions \eqref{NTWS},
\begin{eqnarray}\label{NTWS1}
u(x,t) = \frac{1}{2\lambda_1 t - A}\, \mathrm{exp}\left(-\left| x - \frac{\lambda_2}{2\lambda_1} \ln|2\lambda_1 t - A|- B \right|\right),
\end{eqnarray}
where \( A \) and \( B \) are arbitrary constants.
Let $A=\frac{1}{c}$ and $B=\frac{\lambda_2}{2\lambda_1}\ln|c|$ with $c\neq0$ in \eqref{NTWS1}, we define
\begin{equation*}
u_c(x,t):=\frac{c}{2\lambda_1ct-1}\mathrm{e}^{-\left|x-\frac{\lambda_2}{2\lambda1}\ln|2\lambda_1ct-1|\right|}, ~~c\neq0,~~ t\geq0.
\end{equation*}
Its Fourier transform in $x$ is
\begin{equation*}
\mathscr{F}u_c(\xi,t):=\frac{2c}{2\lambda_1ct-1}\frac{\mathrm{e}^{-i\xi \frac{\lambda_2}{2\lambda_1}\ln|2\lambda_1ct-1|}}{1+\xi^2}.
\end{equation*}
When $p=2$, and $r\in[0,+\infty]$, the Definition \ref{besov} is equivalent to
\begin{eqnarray*}
\|u\|_{B_{2,r}^s}=\left\{
  \begin{array}{ll}
  \left[\left(\int_{-1}^1(1+\xi^2)^s|\hat{u}(\xi)|^2d\xi\right)^{\frac{r}{2}}
+\sum\limits_{q\in\mathbb{N}}\left(\int_{2^q\leq|\xi|\leq2^{q+1}}(1+\xi^2)^s
|\hat{u}(\xi)|^2d\xi\right)^{\frac{r}{2}}\right]^{\frac{1}{r}}, & r\leq\infty, \\
   \max\left\{\left(\int_{-1}^1(1+\xi^2)^s|\hat{u}(\xi)|^2d\xi\right)^{\frac{1}{2}}, \, \sup\limits_{q\in\mathbb{N}} \left(\int_{2^q\leq|\xi|\leq2^{q+1}}(1+\xi^2)^s
|\hat{u}(\xi)|^2d\xi\right)^{\frac{1}{2}}\right\}, & r=\infty.
  \end{array}
\right.
\end{eqnarray*}
For the initial data, we compute $\|u_{c_2}(0)-u_{c_1}(0)\|_{B_{2,\infty}^{3/2}}$ according to the above Definition.
\begin{eqnarray}
\begin{aligned}\label{uc0}
\|u_{c_2}(0)-u_{c_1}(0)\|_{B_{2,\infty}^{3/2}}^2&=8(c_2-c_1)^2\max\left\{\int_{0}^1\frac{1}{\sqrt{1+\xi^2}}d\xi
,\sup_{q\in\mathbb{N}} \int_{2^q}^{2^{q+1}}\frac{1}{\sqrt{1+\xi^2}}d\xi\right\}\\
&=8(c_2-c_1)^2\ln(1+\sqrt{2}).
\end{aligned}
\end{eqnarray}
By direct calculation, we have
\begin{eqnarray}
\begin{aligned}\label{FT2}
&|\hat{u}_{c_2}(t)-\hat{u}_{c_1}(t)|^2=4\left|\frac{c_2}{2\lambda_1c_2t-1}e^{-i\xi\frac{\lambda_2}{2\lambda_1} \ln|2\lambda_1c_2t-1|}-\frac{c_1}{2\lambda_1c_1t-1}e^{-i\xi\frac{\lambda_2}{2\lambda_1} \ln|2\lambda_1c_1t-1|}\right|^2\\ 
=&4\left(\frac{c_2}{2\lambda_1c_2t-1}e^{-i\xi \frac{\lambda_2}{2\lambda_1}\ln|2\lambda_1c_2t-1|}
-\frac{c_1}{2\lambda_1c_1t-1}e^{-i\xi \frac{\lambda_2}{2\lambda_1}\ln|2\lambda_1c_1t-1|}\right)\\
&\times\left(\frac{c_2}{2\lambda_1c_2t-1}e^{i\xi\frac{\lambda_2}{2\lambda_1} \ln|2\lambda_1c_2t-1|}
-\frac{c_1}{2\lambda_1c_1t-1}e^{i\xi\frac{\lambda_2}{2\lambda_1} \ln|2\lambda_1c_1t-1|}\right)
\\=&4\left(\frac{c_2}{2\lambda_1c_2t-1}-\frac{c_1}{2\lambda_1c_1t-1}\right)^2\\
&+\frac{8c_1c_2}{(2\lambda_1c_2t-1)(2\lambda_1c_1t-1)}
\left(1-\cos\left((\ln|2\lambda_1c_2t-1|-\ln|2\lambda_1c_1t-1|)\left|\frac{\lambda_2}{2\lambda_1}\right|\xi\right)\right),
\end{aligned}
\end{eqnarray}
For fixed $T>0$, we choose $c_1=\frac{1}{2\lambda_1(1+T)}$, $c_2=\frac{1+T-e^{2^{(-q+1)}\left|\tfrac{\lambda_1}{\lambda_2}\right|\pi}}{2\lambda_1(1+T)T}$ and $q$ large enough, 
then $c_1c_2>0$ and
\begin{eqnarray}\label{CL}
	\ln|2\lambda_1c_2T-1|-\ln|2\lambda_1c_1T-1|=2^{(-q+1)}\left|\tfrac{\lambda_1}{\lambda_2}\right|\pi.
\end{eqnarray}
Substituting \eqref{FT2} and \eqref{CL} into the above definition of $B_{2,\infty}^{3/2}$, we have
\begin{eqnarray*}
\|u_{c_2}(T)-u_{c_1}(T)\|_{B_{2,\infty}^{3/2}}&\geq& \frac{16c_1c_2}{(2\lambda_1c_2T-1)(2\lambda_1c_1T-1)}\int_{2^q}^{2^{q+1}}\frac{1-\cos(2^{-q}\pi\xi)}{\sqrt{1+|\xi|^2}}d\xi\\
&\geq& \frac{4c_1c_2}{\sqrt{2}(2\lambda_1c_2T-1)(2\lambda_1c_1T-1)},
\end{eqnarray*}
where we have used
\begin{eqnarray*}
\int_{2^q}^{2^{q+1}}\frac{\cos(2^{-q}\pi\xi)}{\sqrt{1+|\xi|^2}}d\xi\leq 0,
\end{eqnarray*}
and
\begin{eqnarray*}
\int_{2^q}^{2^{q+1}}\frac{1}{\sqrt{1+|\xi|^2}}d\xi \geq \frac{2^q}{\sqrt{1+2^{2(q+1)}}} \geq\frac{1}{4\sqrt{2}}.
\end{eqnarray*}
From \eqref{uc0}, we have
\begin{eqnarray*}
\|u_{c_2}(0)-u_{c_1}(0)\|_{B_{2,\infty}^{3/2}}&=&8\left(\frac{1+T-e^{2^{(-q+1)}\left|\tfrac{\lambda_1}{\lambda_2}\right|\pi}}{2\lambda_1(1+T)T}-\frac{1}{2\lambda_1(1+T)}\right)^2\ln(1+\sqrt{2})\\
&=&8\left(\frac{1-e^{2^{(-q+1)}\left|\tfrac{\lambda_1}{\lambda_2}\right|\pi}}{2\lambda_1(1+T)T}\right)^2\ln (1+\sqrt{2}),
\end{eqnarray*}
which implies $\|u_{c_2}(0)-u_{c_1}(0)\|_{B_{2,\infty}^{3/2}}$ may be arbitrary small for $q$ large enough.
Note that for $q$ large enough,
\begin{eqnarray*}
(1+T)\left(1-e^{2^{(-q+1)}\left|\tfrac{\lambda_1}{\lambda_2}\right|\pi}\right)>0,
\end{eqnarray*}
Hence
\begin{eqnarray*}
\frac{1+T-e^{2^{(-q+1)}\left|\tfrac{\lambda_1}{\lambda_2}\right|\pi}}{Te^{2^{(-q+1)}\left|\tfrac{\lambda_1}{\lambda_2}\right|\pi}}>1.
\end{eqnarray*}
Checking that
\begin{eqnarray*}
\frac{c_1}{1-2\lambda_1c_1T}=\frac{1}{2\lambda_1},\quad  \frac{c_2}{1-2\lambda_1c_2T}=\frac{1}{2\lambda_1}\frac{1+T-e^{2^{(-q+1)}\left|\tfrac{\lambda_1}{\lambda_2}\right|\pi}}{Te^{2^{(-q+1)}\left|\tfrac{\lambda_1}{\lambda_2}\right|\pi}}.
\end{eqnarray*}
Therefore, we have
\begin{eqnarray*}
\|u_{c_2}(T)-u_{c_1}(T)\|_{B_{2,\infty}^{3/2}}\geq\frac{1}{2\lambda_1^2}.
\end{eqnarray*}
This completes the proof of Theorem \ref{IPB}.
\end{proof}

\section{Conclusions}

We have derived peakon solutions for a class of Camassa--Holm-type equations with quadratic nonlinearities, which indicates that the convolutional approach is feasible and may be extended to higher-order nonlinear equations. In particular, it is expected to be applicable to more general models, such as the generalized modified Camassa--Holm equation (gmCH) \cite{gao2021global}, for which only existence results are currently available. Non-traveling
wave solutions play a pivotal role in investigating the ill-posedness of Camassa-Holm-type equations with quadratic nonlinearities. Nevertheless, owing to several critical factors—most notably the absence of conserved quantities—the stability analysis \cite{gui2025stability} of non-traveling solitray wave solutions remains highly challenging in three dimensions. Compared with our previous work \cite{zhu2023rogue}, the wave breaking results established in the present paper have been substantially enhanced. 
Some symbiotic bright solitary wave solutions \cite{lin2006symbiotic} might be caught for typical nonlinear dynamical systems, however, additional restrictions are still imposed on the coefficients. Due to intrinsic structural limitations of the equation itself, addressing the wave breaking problem under more general coefficient conditions entails significantly greater technical difficulties.

\end{document}